\documentclass[12pt]{amsart}
\usepackage{graphicx, color}
\newcommand{\intO}{\int \limits_\Omega}
\newcommand{\var}[1] {\intO |D{#1}|}
\newcommand{\Div}{\mathrm{Div}}
\newcommand{\Proj}{\mathrm{Proj}}
\newcommand{\Dist}{\mathrm{Dist}}
\newcommand{\MRR}{M^n \times \mathbb R}
\newcommand{\surfarea}[1] {\intO \sqrt{1 + |D{#1}|^2}}
\newcommand{\loc}{\mathrm{loc}}
\newcommand{\lonorm}[2]{\| {#1} \|_{L^1({#2})}}
\newcommand{\suppt}{\text{suppt }}
\newcommand{\xh}{x_{n+1}}
\newcommand{\la}{\langle}
\newcommand{\ra}{\rangle}
\newcommand{\ue}{u^\epsilon}
\newcommand{\ued}{u^{\epsilon\delta}}
\newcommand{\dx}{\, d x}
\newcommand{\dy}{\, dy}
\newcommand{\dt}{\, dt}

\newcommand{\fe}{F^\epsilon}

\newcommand{\is}{i^*}
\newcommand{\js}{j^*}
\newcommand{\ks}{k^*}

\newcommand{\intAMt}{\int\limits_{A_M(t)}}
\newcommand{\Se}{{S^\epsilon(t)}}
\newcommand{\PhiFn}{\bold \Phi}
\newcommand{\PsiFn}{\bold \Psi}
\DeclareMathOperator{\Lip}{Lip}

\DeclareMathOperator{\Ric}{Ric}

\newtheorem{theorem}{Theorem}[section]
\newtheorem{lemma}[theorem]{Lemma}
\newtheorem{proposition}[theorem]{Proposition}
\newtheorem{corollary}[theorem]{Corollary}
\newtheorem*{introA}{Theorem 6.4}
\newtheorem*{introB}{Theorem 7.3}
\newtheorem*{introC}{Theorem 8.7}

\theoremstyle{definition}
\newtheorem{definition}[theorem]{Definition}

\theoremstyle{remark}
\newtheorem{remark}[theorem]{Remark}

\numberwithin{equation}{section}

\begin{document}
\setlength{\baselineskip}{1.2\baselineskip}

\title[Minimal Graphs and Graphical Mean Curvature Flow in $M^n\times \mathbb R$]
{Minimal Graphs and Graphical Mean Curvature Flow in $M^n\times \mathbb R$ }

\author{Matthew McGonagle}
\address{Department of Mathematics, Johns Hopkins University,
 Baltimore, MD 21218}
\email{mcgonagle@math.jhu.edu}

\author{Ling Xiao}
\address{Department of Mathematics, Johns Hopkins University,
 Baltimore, MD 21218}
\email{lxiao@math.jhu.edu}

\begin{abstract}
In this paper, we investigate the problem of finding
minimal graphs in $M^n\times\mathbb R$ with general boundary conditions using a variational approach. Following Giusti\cite{Giu}, we look at so called generalized solutions that minimize a functional adapted from the area functional. Generalizing results in Schulz-Williams \cite{SW} to $M^n\times\mathbb R,$ we show that for certain conditions on our boundary data $\phi$ the solutions obtain the boundary data $\phi(x)$.  Following Oliker-Ural'tseva \cite{OU1993}, \cite{OU1997} we also consider solutions $\ue$ of a perturbed mean curvature flow \eqref{eq:emcf} for $\epsilon > 0$. We show that there are subsequences $\epsilon_i$ where $u^{\epsilon_i}$ converges to a function $u$ satisfying the mean curvature flow, and subsequences $u(\cdot, t_i)$ converge to a generalized solution $\bar u$ of the Dirichlet problem. Furthermore, $\bar u$ depends only on the choice of sequence $\epsilon_i$.

\end{abstract}

\maketitle

\section*{Introduction}\label{int}
The problem of finding minimal surfaces by using calculus of variations, has been widely studied
in Euclidean space. In particular, in \cite{Giu}, Giusti studies the solvability of the Dirichlet problem
for the minimal surface equation in the space of functions of bounded variation. Giusti also discusses the uniqueness and regularity
of the solution.

For a bounded domain $\Omega \subset M^n$, we study the solvability of the Dirichlet problem
\begin{equation}\label{eq1_1}
\begin{aligned}
\Div\frac{Du}{w}&=0\,\mbox{in $\Omega\subset M^n$,}\\
u&=\varphi\,\mbox{on $\partial\Omega.$}
\end{aligned}
\end{equation}
For boundaries with general mean curvature, solutions to the Dirichlet problem may not exist. To use variational methods, it is then necessary to study an adapted area functional \cite{Giu}:
\begin{equation}\label{eq1_3}
J(u, \Omega) = \mathcal A(u,\Omega) + \int \limits_{\partial \Omega} |u-\varphi| dH_{n-1},
\end{equation}
where $\mathcal A$ is the area functional. In this variational setting, we will easily obtain the existence of bounded local minimizers of
$J(u, \Omega)$ in the class $BV(\Omega),$ the class of functions of bounded variation on $\Omega$. In the space $M^n\times\mathbb R,$ some definitions and key ingredients used in the proofs of Giusti \cite{Giu} differ from the euclidean case. We discuss some of the necessary modifications, and for completeness, include some proofs that are similar to those in Giusti\cite{Giu}. For regularity of solutions, we make use of the a priori interior gradient estimates of Spruck \cite{Spr} for $M \times \mathbb R$.

To consider sufficient conditions for when general solutions to the Dirichlet problem achieve the boundary data $\phi(x)$, we generalize results in Schulz-Williams to $M\times\mathbb R$. Barriers are constructed to show that under certain conditions on $\phi(x)$, we have that any general solution must take on the boundary value $\phi(x).$

Finally, we consider the problem of the graphical mean curvature flow on $M \times\mathbb R$:
\begin{equation}\label{eq1_2}
\begin{cases}
u_t(x,t)  =  WH(u) = \triangle_M u - \frac{u^i u^j}{1+|Du|^2} D^2_{ij} u & \text{in } \Omega\times(0,\infty), \\
u(x,t) = \varphi(x) & \text{on } \partial \Omega \times (0,\infty), \\
u(x,0)  = u_0(x). &
\end{cases}
\end{equation}
As in Oliker-Ural'tseva\cite{OU1993}\cite{OU1997}, we study the perturbed mean curvature flow problem:
\begin{equation}\label{eq:0_4}
\begin{cases}
\ue_t(x,t) = \triangle_M \ue - \frac{u^{\epsilon i} u^{\epsilon j}}{W^2}D^2_{ij}\ue - \epsilon \sqrt{1+|D \ue|^2_M} \triangle_M \ue & \text{in } \Omega\times(0,\infty), \\
\ue(x,t) = \phi(x) & \text{on } \partial\Omega \times (0,\infty), \\
\ue(x,0) = u_0(x).
\end{cases}
\end{equation}
We consider the convergence of the solutions $\ue$ as $\epsilon \to 0$. Specifically, we show that there are sequences $\epsilon_i \to 0$, such that the solutions $u^{\epsilon_i}$ converge in $C^\infty$ on compact subsets to a function $u$ that satisfies mean curvature flow. We also investigate the convergence of $u(\cdot, t)$ to a general solution of the Dirichlet problem as $t\to \infty$.

 The existence of a minimizer of $J(u,\phi)$ in $BV(\Omega)$ is guaranteed by a compactness theorem\ref{fbvt2} for $u\in BV(\Omega)$. Our first main result is a statement on the interior regularity of minimizers.
 \begin{introA}\label{th1_1}
 Let $u\in BV(\Omega)\cap L^{\infty}(\Omega)$ be a local minimizer to \eqref{eq1_3}. Then $u\in C^{\infty}(\Omega).$
 \end{introA}
By constructing appropriate barriers, we show
\begin{introB}\label{the1_2}
Let $\Omega$ be a bounded open subset of $M^n$ with $C^2$ boundary $\partial\Omega.$
Given $n\geq 2,$ $K\in\left(0, \frac{1}{\sqrt{(n-1)\gamma}}\right),$ and $\gamma>1.$
Suppose $\varphi$ is Lipschitz continuous on $\partial\Omega$ and
\begin{equation}|\varphi(x)-\varphi(y)|\leq K\|x-y\|_{M^n},\,\,\mbox{if $x, y\in\partial\Omega$},\end{equation}
\begin{equation}\sup\limits_{\partial\Omega}\varphi(x)-\inf\limits_{\partial\Omega}\varphi(x)\leq\epsilon,\end{equation}
where $\epsilon$ depends on $n,$ $K,$ $M^n,$ and $\partial\Omega.$
Then there exists a function $u\in C^2(\Omega)\cap C^0(\bar\Omega)$ such that u is the solution of the equation \eqref{eq1_1}.
\end{introB}
Our last main result is an application to the graphical mean curvature flow.
\begin{introC}\label{the1_3}
Let $\Ric_M \geq 0$ and let $\Omega \subset M$ be a bounded domain. There exists a sequence $\epsilon_i \to 0$ such that the solutions $u^{\epsilon_i}$ of \eqref{eq:0_4} converges uniformly to a function $u$ in $C^\infty(K)$ for compact $K \subset\subset \Omega\times (0,\infty)$. We have that $u\in C^\infty(\Omega\times (0,\infty)) \cup L^\infty([0,\infty); W^{1,1}(\Omega)).$

Consider any sequence $t_i$ such that $u(\cdot, t_i)$ converges uniformly to a function $\bar u$ in $C^\infty(K)$ for compact $K \subset\subset \Omega$. We have that $\bar u$ is a generalized solution to the Dirichlet problem with boundary data $\phi(x)$. Furthermore, $\bar u$ depends only on the choice of sequence $\epsilon_i$; $\bar u$ is independent of the choice of $t_i$.
\end{introC}

\section*{Acknowledgments}
We would like to thank Professor Spruck for his guidance and suggestions.

\medskip
\section{Functions of Bounded Variation}\label{fbv}


For any vector bundle $V$ over $M$, we will denote the $C_0^1$ sections of $V$ over an open set $\Omega \subset M$ by
$\Gamma_0^1 (V)$.

\begin{definition}\label{fbvd1}
Let $\Omega$ be a bounded open subset of $M$. The variation of a measurable function $f$ is defined by
\begin{equation}
\var f \equiv \sup \{  \intO f \Div g\ : g \in \Gamma_0^1(T\Omega) \text{ and } |g| \leq 1 \text{ pointwise } \}.
\end{equation}
A function $f$ with $\var f < \infty$ is called a function of bounded variation on $\Omega$. We abbreviate this as $f \in BV(\Omega)$. We will also use $\| f\|_{BV(\Omega)} \equiv |f|_{L^1(\Omega)} + \var f$.
\end{definition}
 Note, that by the Riesz Representation Theorem, we have that for test functions $g$ with support contained in a local coordinate chart, there exists $\eta \in T\Omega$ with $|\eta| = 1$ and
\begin{equation}
\intO f \Div g = \intO g_{ij} \eta^i g^j d|Df|.
\end{equation}
Note, that $\eta$ is actually independent of coordinates, and so gives us a well defined $\eta \in T\Omega$. For $X \in C^1(T\Omega)$ (not necessarily compactly supported) we may then define
\begin{equation}
\intO \la Df, X \ra \equiv \intO g_{ij} \eta^i X^j \, d|Df|
\end{equation}

To extend the notion of the surface area of the graph of a function $f\in C^{0,1}(M)$ to functions $f\in BV(\Omega)$, we make an additional definition.

\begin{definition}\label{fbvd2}
Let $\Omega$ be a bounded open subset of $M$ and $f \in BV(\Omega)$. We define
\begin{equation}
\surfarea{f} \equiv \sup \{  \intO g_{n+1}+f\Div g :g \in \Gamma_0^1(T\Omega),
\end{equation}
\begin{equation}
g_{n+1} \in C_0^1(\Omega), \text{ and } |(g, g_{n+1})| \leq 1 \text{ pointwise }  \}.
\end{equation}
We willl also denote this quantity by $\mathcal{A}(u,\Omega)$.
\end{definition}

Clearly, we have that
\begin{equation}
\var f \leq \surfarea f \leq \var f + |\Omega|.
\end{equation}
Also, if $f \in W^{1,1}(\Omega)$, then
\begin{equation}
\surfarea f = \intO \sqrt{1 + |\nabla f|^2}.
\end{equation}


\begin{lemma}[Lower Semi-Continuity]\label{fbvl1}
If $u_j \to u$ in $L^1_\loc (\Omega)$ then
\begin{equation}
 \surfarea u \leq \liminf \limits_{j\to\infty} \surfarea{u_j} .
\end{equation}
\end{lemma}

\begin{proof}
For any $(g, g_{n+1}) \in \Gamma_0^1 (T\Omega \times \mathbb R),$ we have
\begin{equation}
 \intO g_{n+1} + u\Div g = \lim_j \intO g_{n+1} + u_j \Div g \leq \liminf_j \surfarea{u_j}.
\end{equation}
\end{proof}

\begin{remark}
Note, that we don't get continuity. Consider the sets $S_i = [-2,2]^2 - ([-1/i, 1/i]\times[-1,1]) \subset \mathbb R^2$ and the functions $f_i = \chi_{S_i}$.
\end{remark}


An argument made in Giusti \cite[Theorem 1.17]{Giu} shows that

\begin{theorem}\label{fbvt1}
Let $\Omega$ be a bounded open set in $M$ and $f \in BV(\Omega)$. Then, there exists a sequence $f_j \in C^\infty(\Omega)$ such that $f_j \to f$ in $L^1(\Omega)$ and $\var{f_j} \to \var{f}$.
\end{theorem}

From this we may show a generalization of the compactness theorem.

\begin{theorem}\label{fbvt2}
Let $\Omega \subset M$ be a bounded open subset that is sufficiently regular for the Rellich Theorem to hold. Then, sets of functions uniformly bounded in $BV(\Omega)$ are relatively compact in $L^1(\Omega)$.
\end{theorem}

\begin{proof}

Let $f_j \in BV(\Omega)$ such that $\| f_j \|_{BV(\Omega)} \leq M$. By Theorem \ref{fbvt1}, for each $j$, we may find $g_j \in C^\infty(\Omega)$ such $\intO |f_j - g_j| < j/2$ and $\var{g_j} < M+2$. So, our sequence $g_j$ is uniformly bounded in a Sobolev Space, and we may apply the Rellich Theorem to obtain a subsequence of $f_j$ converging to a function $f \in L^1(\Omega)$.

Now, using Lemma \ref{fbvl1} we can see that $f \in BV(\Omega)$.

\end{proof}

\begin{remark}
A sufficient condition for the hypotheses of \ref{fbvt2} to hold is that $\partial \Omega$ is Lipschitz-continuous.
\end{remark}


By using partitions of unity, we may extend \cite[Theorem 2.16]{Giu} from the Euclidean case to $M^n\times \mathbb{R}$.

\begin{theorem}[Extension of the Boundary]\label{fbvt3}
Let $\Omega \subset M$ be a bounded open set with Lipschitz-continuous boundary $\partial \Omega$ and let $\varphi \in L^1(\partial \Omega)$. For every $\epsilon > 0$ there exists $f \in W^{1,1}(\Omega)$ having trace $\varphi$ on $\partial \Omega$ such that
\begin{equation}
 \intO |f| \leq \epsilon \lonorm{\varphi}{\partial \Omega}
\end{equation}
\begin{equation}
 \var f \leq A \lonorm{\varphi}{\partial \Omega}
\end{equation}
with $A$ depending only on $\partial \Omega$.
\end{theorem}
\begin{remark}\label{fbvr3}
In this paper, unless otherwise stated, we always assume $\partial\Omega$ is Lipschitz-continuous.
\end{remark}

\begin{remark}
\label{fbvr1}
The construction of $f$ (see Giusti\cite[Chapter 2]{Giu}) shows that if $\partial\Omega$ is of class $C^1,$ we may take $A=1+\epsilon.$
\end{remark}
We may also generalize a result in Giusti\cite{Giu} relating the minima of two different variational problems.

\begin{theorem}
\label{fbvt4}
Let $\Omega$ be a bounded open set with $C^1$ boundary $\partial \Omega \subset M$. Also, let $\varphi \in L^1(\partial \Omega)$. We have that
\begin{align}
 &\inf \{ \mathcal{A}(u,\Omega) : u \in BV(\Omega), u=\varphi \text{ on } \partial \Omega  \} \\
 = &\inf \{ \mathcal{A}(u,\Omega) + \int \limits_{\partial \Omega} |u-\varphi| dH_{n-1} : u \in BV(\Omega) \}.
\end{align}
\end{theorem}

\begin{proof}
It is clear that the left side is ``$\geq$" to the right side. So, we will now show the opposite inequality.

Given $\epsilon > 0$ and any $u\in BV(\Omega)$, from Theorem \ref{fbvt3}, we have that there exists $w \in W^{1,1}(\Omega)$
such that
\begin{equation}
\var{w} \leq (1+\epsilon) \lonorm{\varphi - u}{\partial \Omega}
\end{equation}
and \[w = \varphi - u \mbox{\,\,on $\partial \Omega$}.\]
Now, define $v \equiv u + w$. Note that $v \in BV(\Omega)$ and that $v = \varphi$ on $\partial \Omega$.

We see that
\begin{align}
\surfarea{v} & \leq \surfarea{u} + \var{w} \\
& \leq \surfarea{u} + (1+\epsilon)\lonorm{u-\varphi}{\partial\Omega}.
\end{align}
Taking $\epsilon \to 0$, we get our desired inequality and the theorem.
\end{proof}


We also have a lemma for constructing functions of bounded variation by gluing together functions defined on adjacent domains.

\begin{lemma}
\label{fbvl2}
Choosing $R>0$ sufficiently small such that there exists local coordinates in $\mathbf B_{3R}$ on $M\times\{0\}$, and
let $C_R = \mathcal B_R \times (-R,R)\subset\mathbf B_{3R}\subset M\times\{0\}$.
Denote the upper half cylinder by $C_R^+ = \mathcal B_R \times (0,R)$
and the lower cylinder by $C_R^- = \mathcal B_R \times (-R,0)$ so that $\mathcal B_R = \partial C_R^+ \cap \partial C_R^-$. Given functions $f_1 \in BV(C_R^+)$ and $f_2 \in BV(C_R^-)$, define a function
$f = \begin{cases}
f_1 & \text{in } C_R^+ \\
f_2 & \text{in } C_R^-
\end{cases}$. Then, $f \in BV(C_R)$ and
\begin{equation}
\int\limits_{\mathcal B_R} |Df| = \int\limits_{\mathcal B_R} |f^+ - f^-|,
\end{equation}
where $f^+$ is the trace of $f_1$ on $\mathcal B_R,$ and $f^-$ is the trace of $f_2$ on $\mathcal B_R.$
\end{lemma}

\begin{proof}
\begin{equation}\label{fbv1}
\begin{aligned}
\int_{\mathcal C_R}f \Div g= &-\int_{\mathcal C_R^+}\langle g,Df\rangle+\int_{\mathcal B_R}f^+\langle g,\nu\rangle d_{H_n}\\
&-\int_{\mathcal C_R^-}\langle g,Df\rangle-\int_{\mathcal B_R} f^-\langle g,\nu\rangle d_{H_n}.
\end{aligned}
\end{equation}
On the other hand
\begin{equation}\label{fbv2}
\int_{\mathcal C_R}f\Div g=-\int_{\mathcal C_R^+}\langle g, Df\rangle-\int_{C_R^-}\langle g, Df\rangle
-\int_{\mathcal B_R}\langle g,Df\rangle.
\end{equation}
Therefore,
\begin{equation}
-\int_{\mathcal B_R}\langle g,Df\rangle=\int_{\mathcal B_R}(f^+-f^-)\langle g,\nu\rangle d_{H_n}.
\end{equation}
\end{proof}

\begin{remark}
\label{fbvr2}
By using a partition of unity argument as in Remark 2.9 of \cite{Giu}, we have that if $A \subset\subset \Omega \subset\subset M^n$ is an open set with Lipschitz continuous boundary $\partial A$, then $f|_A$ and $f|_{\Omega\backslash \bar A}$ will have traces on $\partial A$ which we will call $f_A^-$ and $f_A^+$ respectively. Then $\int\limits_{\partial A} |f_A^+ - f_A^-| = \int\limits_{\partial A} |Df|.$
\end{remark}


Like Sobolev functions, functions of bounded variation have natural trace operators for giving the values of a function $f\in BV(\Omega)$ on $\partial\Omega$. \cite{Giu}

The reduced boundary $\partial^* E$ is intuitively defined to be the regular part of the boundary. $\partial^* E$ is specifically defined by $x\in \partial^* E$ if and only if
\begin{enumerate}
\item$\int\limits_{B(x,\rho)} |D\phi_E| > 0$ for all $\rho > 0$,
\item The limit $\nu(x) = \lim\limits_{\rho\to 0}\frac{\int_{B(x,\rho)} D\phi_E}{\int_{B(x,\rho)}|D\phi_E|}$ exists,
\item $|v(x)|$ = 1.
\end{enumerate}
The singular set of a set $E$ is specifically defined to be $\partial E \setminus \partial^* E$.

A Cacciopoli set $F$ is defined to be a Borel set $F$ such that for every bounded domain $\Omega \subset M\times\mathbb R$ we have $\int\limits_{\Omega} |D\phi_F| < \infty$. For a Cacciopoli set $E$, up to set of negligible measure, we have that $\partial^* E$ is a countable union of compact $C^1$ hypersurfaces \cite[Theorem 4.4]{Giu}.


\section{functionals}\label{fun}

\begin{definition}\label{fund1}
We define
\begin{equation}
J(u, \Omega) = \mathcal A(u,\Omega) + \int \limits_{\partial \Omega} |u-\varphi| dH_{n-1}.
\end{equation}
\end{definition}

Instead of directly solving the Dirichlet boundary value problem, we solve the variational problem of finding $u \in BV(\Omega)$ minimizing $J(w,\Omega)$ among all $w \in BV(\Omega)$. Note, since boundary values are not preserved by limits in $L^1$, we are using the boundary integral to penalize not matching the boundary values of the Dirichlet problem.

\begin{remark}\label{funr2} Let $\Omega$ be a bounded open set with Lipschitz continuous boundary. If $\mathfrak B$
is a ball containing $\bar\Omega$ we can use Theorem \ref{fbvt3} to extend $\varphi$ to a $W^{1,1}$ function in $\mathfrak B-\bar\Omega,$
that we will denote again by $\varphi.$ If we set for $v\in BV(\Omega)$
\begin{equation}
v_{\varphi}(x)=\begin{cases}
v(x) & x\in\Omega\\
\varphi(x) & x\in\mathfrak B-\Omega\\
\end{cases}
\end{equation}
then $v_{\varphi}\in BV(\mathfrak B)$ and
\begin{equation}\label{fun19}
\int\limits_{\mathfrak B}\sqrt{1+|Dv_\varphi|^2}=J(v, \Omega)+\int\limits_{\mathfrak B-\bar\Omega}\sqrt{1+|D\varphi|^2}dx.
\end{equation}
In the future, we denote
\begin{equation}
\mathcal A_\varphi(v, \mathfrak B)=\int\limits_{\mathfrak B}\sqrt{1+|Dv_\varphi|^2}.
\end{equation}
It is clear that $u$ is a minimizer for our Dirichlet problem for $J(v,\Omega)$ if and only if $u_\varphi$ is a minimizer to the equivalent problem: Given a function $\varphi\in W^{1,1}(\mathfrak B-\bar\Omega),$
find a function $u\in BV(\mathfrak B),$ coinciding with $\varphi$ in $\mathfrak B-\bar\Omega$ and minimizing the area
$\mathcal A_\varphi(v; \mathfrak B).$
\end{remark}

\begin{theorem}[Existence of a Minimizer]\label{funt1}
Let $\Omega$ be a bounded open set in M with Lipschitz boundary $\partial \Omega$, and let $\varphi$ be a function in $L^1(\partial \Omega)$. The functional $J(u, \Omega)$ attains its minimum in $BV(\Omega)$.
\end{theorem}

\begin{proof}
Applying Theorem \ref{fbvt2} and Lemma \ref{fbvl1} to the equivalent problem for $A_\varphi(u,\mathfrak B)$, we get the conclusion.
\end{proof}


Just like in Giusti\cite{Giu}, we have a theorem relating the surface area of the graph of a function of bounded variation to the variation of the characteristic function
of its subgraph.

\begin{theorem}\label{funt2}
Let $u \in BV(\Omega)$ and let $U = \{(x,t) \in \Omega \times \mathbb R : t < u(x)  \}$ be the subgraph of $u$. We have that
\begin{equation}\label{fun1}
\surfarea{u} = \int \limits_{\Omega \times \mathbb R} |D\phi_U|,
\end{equation}
where $\phi_U$ is the characteristic function of set $U.$
\end{theorem}

\begin{proof}

\newcommand{\intOR}{\int \limits_{\Omega \times \mathbb R}}

First, consider the case that $u$ is bounded. By translating we may consider $u \geq 1$. Let $g(x)\in\Gamma^1_0(T\Omega)$ and $g_{n+1}(x) \in C_0^1(\Omega)$  such that $|(g(x),g_{n+1}(x))| \leq 1$. Let $\eta(t)$ be a function such that $\suppt \eta \subset [0, 1+\sup_\Omega u] $, $\eta \equiv 1$ on $[1, \sup_\Omega u]$, and $|\eta| \leq 1$. Let $H(x,x_{n+1}) = \eta(x_{n+1})(g(x),g_{n+1}(x))$. We have $|H| \leq 1$ in $\Omega \times \mathbb R$.
Then, note that
\begin{align}\label{fun2}
\intOR |D\phi_U| & \geq \int_U  \Div H \\
 & = \intO \int \limits_0^{u(x)} g_{n+1} \eta'(\xh) + \eta(\xh)\Div g \, dx_{n+1}.
 \end{align}
Since
\begin{equation}\label{fun3}
\int \limits_0^ {u(x)} \eta'(\xh) d\xh = 1,
\end{equation}
and
\begin{align}\label{fun4}
\int \limits_0^{u(x)} \eta(\xh) d\xh & = u(x) - \int \limits_0^1 (1- \eta(\xh))d\xh\\
& = u(x) - C.
\end{align}
 Therefore,
 \begin{equation}\label{fun5}
 \intOR |D\phi_U| \geq \intO g_{n+1} + u \Div g.
 \end{equation}
 Hence, we get that
 \begin{equation}\label{fun6}
\intOR |D\phi_U| \geq \intO \sqrt{1 + |Du|^2}.
\end{equation}

To prove the opposite direction, we first note that we have equality for $C^1$ functions. Now, approximate $u \in BV(\Omega)$ by $C^1$ functions such that $u_j \to u$ in $L^1(\Omega)$ and $\surfarea{u_j}\to \surfarea u$. On the other hand
we have that
\begin{equation}
 \phi_{U_j} \to \phi_U \,\mbox{in}\, L^1_{loc} (\Omega \times \mathbb R),
\end{equation} and therefore
\begin{align}\label{fun7}
\intOR |D\phi_U|  \leq \liminf_{j\to\infty} \intOR |D\phi_{U_j}| & = \lim_{j\to\infty} \surfarea{u_j} \\
& = \surfarea{u}.
\end{align}

For $u$ unbounded, we set

\begin{equation}
 u_T(x) = \begin{cases}
		u(x) & \mbox{if}\, |u|< T \\
		T & \mbox{if}\, u \geq T \\
		-T & \mbox{if}\, u \leq -T
\end{cases}\end{equation}

Letting $T \to \infty$, we get the result.
\end{proof}

Next, following Giusti\cite{Giu}, we show that a sort of vertical rearrangement of certain type of set $F$ produces the graph of a function with no more perimeter than $F$.

\begin{lemma}\label{funl1}
Let $F \subset \Omega \times \mathbb R$ be measurable, and suppose that for some $T>0$ we have
\begin{equation} \Omega \times (-\infty, -T) \subset F \subset \Omega\times (-\infty, T).\end{equation}
For $x\in \Omega$ let
\begin{equation}w(x) \equiv \lim \limits_{k\to\infty} (\int \limits_{-k}^k \phi_F(x,t)dt - k),\end{equation}
then
\begin{equation}\label{fun8}
\surfarea{w} \leq \int \limits_{\Omega\times\mathbb R} |D\phi_F|.
\end{equation}
\end{lemma}

\begin{figure}
\def\svgwidth{3 in}
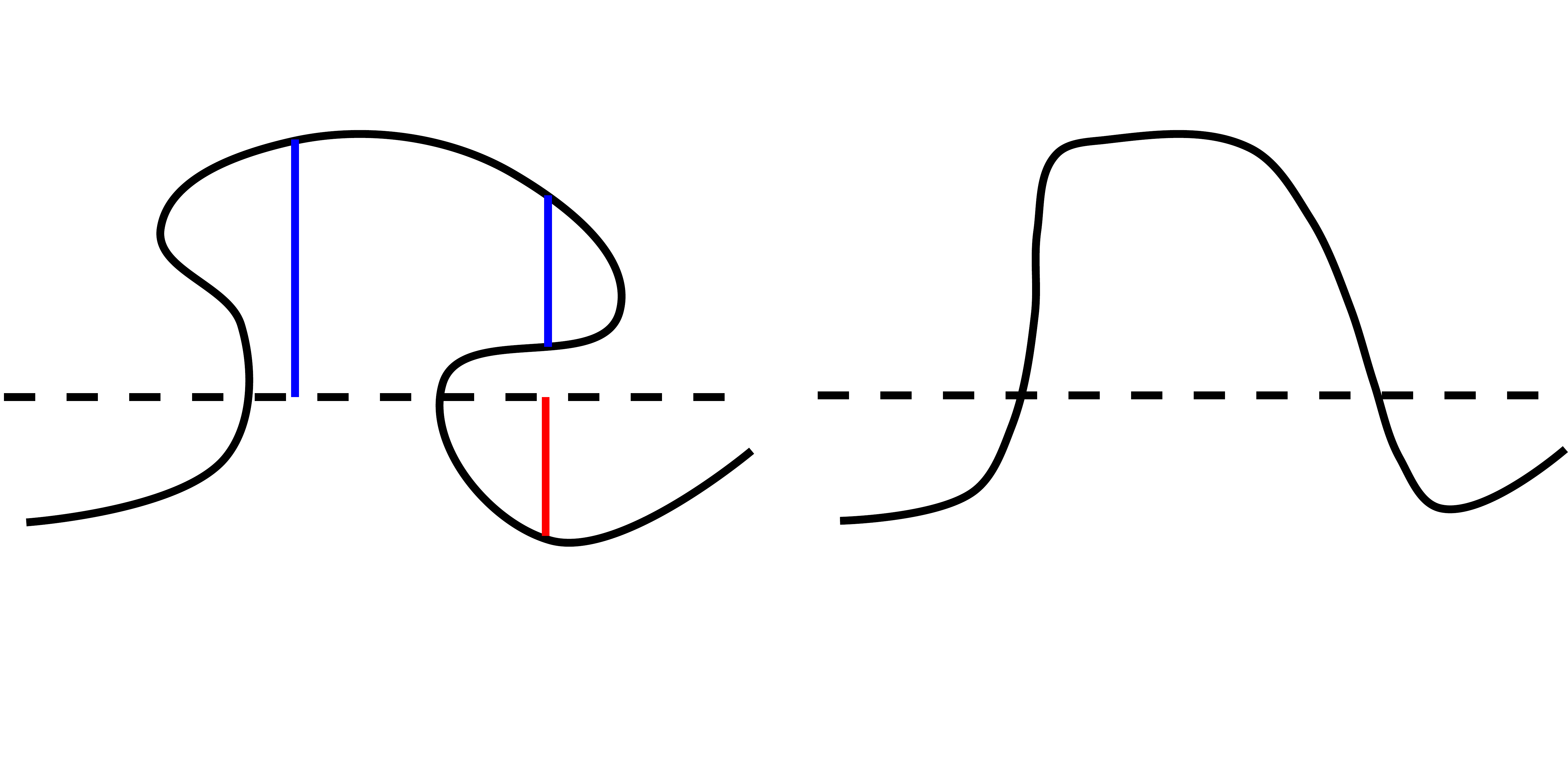
\end{figure}

Here the function $w(x)$ is essentially adding up the amount of $F$ above $(x,0)$ and subtracting the amount of $(M\times\mathbb R)\setminus F$ below $(x,0)$. Thus, we see that the subgraph of $w(x)$ is a sort of vertical rearrangement for $F$ which reduces the area. The mass over or below $(x,0)$ is redistributed to continuously start from the lowest boundary point of $F$ over or below $(x,0)$. We say ``sort of," because $H^{n+1}(F) = \infty$. Note that if $F$ is the subgraph of function $u(x)$ then $w(x) = u(x)$.

\begin{proof}
We first note that it is clear from our conditions on $F$ that $-T \leq w \leq T$.

Now, let $g(x) \in \Gamma_0^1(T\Omega)$ and $g_{n+1}(x) \in C_0^1(\Omega)$ such that $|(g(x),g_{n+1}(x))| \leq 1$. Also, let $\eta(t)$ be a smooth function such that $0 \leq \eta \leq 1$, $\eta(t) = 0 $ if $|t|\geq T+1$, and $\eta(t) = 1$ if $|t| \leq T.$

From our assumptions on $F$, we have that
\begin{equation}\label{fun9}
\int\limits_{-\infty}^\infty \eta'(t)\phi_F(x,t)dt = 1,
\end{equation}
and that
\begin{equation}
w(x) = \int\limits_{-T}^{T} \phi_F(x,t) \dt - T.
\end{equation}
So,
\begin{equation}\label{fun10}
\int\limits_{-\infty}^\infty \eta(t)\phi_F(x,t) dt = w(x) + T + \int\limits_{-T-1}^{-T} \eta(t) dt = w(x) + \alpha_T.
\end{equation}

Hence,
\begin{align}\label{fun11}
\int\limits_{\Omega\times\mathbb R} |D\phi_F| & \geq \int\limits_{\Omega\times\mathbb R} \phi_F(x,x_{n+1})\Div_{M\times\mathbb R}(\eta(x_{n+1})(g,g_{n+1}))\\
&  = \intO (w+\alpha_T)\Div g + g_{n+1}  \\
& = \intO w \Div g + g_{n+1}.
\end{align}
Taking the supremum over $|(g(x),g_{n+1}(x))| \leq 1,$ we get the result.
\end{proof}

\begin{theorem}\label{funt3}
Let $F$ be a measurable set in $Q \equiv \Omega \times \mathbb R$ such that
\begin{enumerate}
\item For a.e. $x\in \Omega$ we have
\begin{align}
\lim \limits_{t\to \infty} \phi_F(x,t) & = 0 \\
\lim \limits_{t\to -\infty}\phi_F(x,t) & = 1.
\end{align}

\item The symmetric difference $F_0 = (F-Q^-)\cup (Q^- - F)$ has finite measure, where $Q^- \equiv \{ (x,t) \in Q : t<0 \}.$
\end{enumerate}
Then the function $w(x) = \lim \limits_{k\to\infty} (\int \limits_{-k} ^{k} \phi_F(x,t)dt - k)$ belongs to $L^1(\Omega),$
 and
\begin{equation}\label{fun12}
\surfarea{w} \leq \int \limits_Q |D\phi_F|.
\end{equation}
\end{theorem}

\begin{remark}\label{funr1}
In the case that $\Omega \subset \mathbb R^n$, conditions (1) and (2) are redundant for Caccioppoli sets F, because (2) follows from (1) by an isoperimetric inequality for $\mathbb R^n$. Moreover, in general, the redundancy depends on the existence of an isoperimetric inequality.
\end{remark}

\begin{proof}
We define
\begin{equation}F_k \equiv F\cup [\Omega\times(-\infty,-k)] - [\Omega\times(k,\infty)].\end{equation}
Let
\begin{equation}\label{fun13}
w_k \equiv \int\limits_{-k}^k \phi_F(x,t)dt - k.
\end{equation}
Note that $w_k \to w$ pointwise and $w_k = |\{t: (x,t) \in F_k \setminus Q^-\}| - |\{t: (x,t) \in Q^- \setminus F_k\}|$. Denote
\begin{equation}\label{fun14}
f(x) = |\{t : (x,t) \in F_0\}|.
\end{equation}
From our second hypothesis we have that $f \in L^1(\Omega)$. We also see that $|w_k| \leq f$. Therefore, by dominated convergence, we have that $w_k \to w$ in $L^1(\Omega)$.

Note that,
\begin{equation}\label{fun15}
w_k(x) = \lim\limits_{l\to\infty} \int\limits_{-l}^l\phi_{F_k} (x,t) dt - l.
\end{equation}
Hence, from the preceding lemma, we get that
\begin{align}\label{fun16}
\surfarea{w_k} & \leq \int \limits_Q |D\phi_{F_k}| \\
& \leq \int \limits_{\Omega\times(-k,k)} |D\phi_{F} |
+\int \limits_{\Omega\times\{k\}} \phi_F dM + \int \limits_{\Omega\times\{-k\}} 1-\phi_F dM.
\end{align}
Taking $k\to\infty$ and using lower semi-continuity we get the result.
\end{proof}

\begin{theorem}
Let $u\in BV_{loc}(\Omega)$ be a local minimum of the area functional. Then the set $U = \{ (x,t) \in \Omega \times \mathbb R : t < u(x)\}$ minimizes locally the perimeter in $Q = \Omega \times \mathbb R$.

\end{theorem}

\begin{proof}
Let $A \subset\subset \Omega$ and let $F$ be any set in $Q$ coinciding with $U$ outside a compact set $K \subset A\times\mathbb R$. Since $u$ is in $L^1(\Omega)$ we have that $U$ satisfies conditions (1) and (2) of Theorem \ref{funt3}.

Now, since $F$ differs from $U$ on a compact set, we have that $F$ satisfies conditions (1) and (2) of Theorem \ref{funt3}. As in Theorem \ref{funt3}, we define
\begin{equation}\label{fun17}
w(x) = \lim\limits_{k\to\infty}\int\limits_{-k}^k \phi_F(x,t)dt-k.
\end{equation}
Since $w$ coincides with $u$ outside $A$ and $\Omega\setminus A$ is open, we have that $\int\limits_{\Omega\setminus A} \sqrt{1 + |Du|^2} \dx = \int\limits_{\Omega\setminus A} \sqrt{1+|Dw|^2}$. That is, we don't pick up any extra area since $\partial A \setminus A = \emptyset$. Since $u$ locally minimizes the area functional, we then get that
\begin{equation}\label{fun18}
\int\limits_{A\times\mathbb R}|D\phi_U| = \int\limits_A \sqrt{1+|Du|^2} \leq \int\limits_A \sqrt{1+|Dw|^2} \leq \int\limits_{A\times\mathbb R}|D\phi_F|.
\end{equation}
\end{proof}


\section{Perimeter}\label{per}

\begin{definition}\label{perd1}
For a Cacciopoli set $E$ and $\Omega$ an open domain, we define $P(E,\Omega) = \var{\phi_E}.$ Note, that $P(E,\Omega)$ is the perimeter of $E$ inside $\Omega$, but not including the perimeter coming from $\partial \Omega$.
\end{definition}

\begin{lemma}\label{perl1}

For a Cacciopoli set $E$ and $\rho < R$, we have that\\
i) $P(E\cap B_\rho, B_R) = P(E,B_\rho) + H_{n-1}(\partial B_\rho \cap E)$, \\
ii)$P(E\backslash B_\rho, B_R) = P(E, B_R \backslash B_\rho) + H_{n-1}(\partial B_\rho \cap E).$

\end{lemma}

\begin{proof}
From Remark \ref{fbvr2}, we conclude that if we define $F: \Omega \to \mathbb R$ by
\begin{equation}\label{per1}
F(x) = \begin{cases} f(x) & x\in A \\
0 & x \in \Omega\backslash A
\end{cases}.\end{equation}

then
\begin{equation}\label{per2}
\var F = \int\limits_A |Df| + \int\limits_{\partial A \cap \Omega}|f_A^-|.
\end{equation}

For equation (i), we let $A = B_\rho$ and $\Omega = B_R$. Also, define
\begin{equation}\label{per3}
F = \phi_{E\cap B_\rho} = \begin{cases} \phi_E & x\in B_\rho \\ 0 & x\in \Omega \backslash B_\rho \end{cases}.
\end{equation}

We then get that
\begin{equation}\label{per4}
\begin{aligned}
\int\limits_{B_R}|D\phi_{E\cap B_\rho}|&= \int\limits_{B_\rho} |D\phi_E| + \int\limits_{\partial B_\rho \cap B_R}|\phi_E^-| dH_{n-1}\\
 &= P(E,B_\rho) + H_{n-1}(\partial B_\rho \cap E).
\end{aligned}
\end{equation}

Now for equation (ii), let $A = B_R \backslash \bar B_\rho$, $\Omega = B_R$, and
\begin{equation}\label{per5}
F = \phi_{E\backslash \bar B_\rho} = \begin{cases} \phi_E & x\in B_R \backslash \bar B_\rho \\ 0 & x\in B_\rho \end{cases}.
\end{equation}

Then, we have that
\begin{equation}\label{per6}
\int\limits_{B_R}|D\phi_{E \backslash B_\rho}| = \int\limits_{B_R \backslash B_\rho} |D\phi_E| + \int\limits_{\partial(B_R \backslash B_\rho)\cap B_R}|\phi_E^-| dH_{n-1}.
\end{equation}
Note that $\partial(B_R \backslash B_\rho)\cap B_R = \partial B_\rho$, and we get the lemma.
\end{proof}

We also recall without proof some properties of perimeter (see Miranda\cite{Mir}).

\begin{lemma}\label{perl2}
i) (Locality) $ P(E,A) = P(F,A)$ whenever $Vol((E\triangle F)\cap A) = 0$, \\
ii) (Subadditivity) $P(E\cup F,A) + P(E\cap F, A) = P(E,A) + P(F,A)$,\\
iii) (Complementation) $ P(E,A) = P(E^c, A)$.
\end{lemma}

By applying Lemma \ref{perl2} we have
\begin{lemma}[Absolute Continuity]\label{perl3}
If $E$ is a set of finite perimeter in $M^{n+1}$, then we have that $P(E,B) = 0$ whenever $H_n(B) = 0$.
\end{lemma}

\begin{proof}
We may assume without loss of generality that $B$ is a compact set. Let $\mathfrak B_r$ be a geodsic ball cenetered at some point. Hence, for some constant $C = C(n)$, there exists an $R$ such that for all $r < R$ we have that $H_n(\partial \mathfrak B_r) < C r^n$ and $H_{n+1}(\mathfrak B_r) < C r^{n+1}$. Since $H_n(B) = 0$, for any $\epsilon > 0$, we can cover $B$ in a finite number of balls $\mathfrak B^\epsilon_i$ of radius $r_i^\epsilon$ and center $x_i^\epsilon$ such that $\sum\limits_i (r_i^\epsilon)^n < \epsilon.$ Now, let $A_\epsilon \equiv \bigcup \limits_i B_{2r_i^\epsilon}(x_i^\epsilon).$ We have $B \subset \subset A_\epsilon$ and
\begin{equation}\label{per7}
P(A_\epsilon, M^{n+1}) \leq C \sum\limits_i (2r_i^\epsilon)^n \leq 2^nC\epsilon.
\end{equation}
So, by Lemma \ref{perl2} we have
\begin{equation}\label{per8}
\begin{aligned}
P(E\cup A_\epsilon, M^{n+1})&= P(E \cup A_\epsilon, M^{n+1}\backslash B)\\
&\leq P(E,M^{n+1}\backslash B) + P(A_\epsilon, M^{n+1})\\
& \leq P(E, M^{n+1}\backslash B) + 2^n C\epsilon.
\end{aligned}
\end{equation}

Since, $H_{n+1}(B) = 0$ as well, we have that $H_{n+1}(A_\epsilon) \to 0.$ Therefore, upon passing to the limit as $\epsilon \searrow 0$, the lower semi-continuity of perimeter gives
\begin{equation} P(E, M^{n+1}) \leq P(E, M^{n+1}\backslash B).\end{equation}
Hence, $P(E,B) = 0$.
\end{proof}


\section{Sobolev Inequality and Consequences}\label{sic}

We first have the Sobolev inequality of D. Hoffman and J. Spruck\cite{HS}.
\begin{theorem}(See \cite[Theorem 2.2]{HS})\label{sict1}
Let $M\subset N$ be compact with boundary $\partial M$ and assume that $K_N \leq b^2$. Then
\begin{equation}\label{sic1}
|M|^{(m-1)/m} \leq C(m,\alpha) \left(\text{Vol} (\partial M) + \int\limits_M |H| dV_M\right),
\end{equation}
provided that
\begin{equation}\label{sic2}
b^2 (1-\alpha)^{-2/m}(\omega_m^{-1} \text{Vol} M)^{2/m} \leq 1
\end{equation}
and
\begin{equation}\label{sic3}
2\rho_0 \leq \bar R(M).
\end{equation}
Here, $0 < \alpha < 1$,
\begin{equation}\label{sic4}
\rho_0 = \begin{cases}
 b^{-1} \sin^{-1} b (1-\alpha)^{-1/m}(\omega_m^{-1} \text{Vol}M)^{1/m} & \text{for b real} \\
(1-\alpha)^{-1/m}(\omega_m^{-1} \text{Vol} M)^{1/m} & \text{for b imaginary} \end{cases},
\end{equation}
and $\bar R(M) = $minimum distance to the cut locus in $N$ for all points in $M$. Finally,
\begin{equation}
C(m,\alpha) = C(m) \alpha^{-1}(1-\alpha)^{-1/m}.
\end{equation}
\end{theorem}

So, we see that for the application of the Sobolev Inequality, as $\text{Vol}(M)$ gets larger, we must reduce $1-\alpha$ which causes worse values of $C(m,\alpha)$. Next, we have a proposition giving estimates for the intersection of minimal sets with balls.

\begin{proposition}\label{sicp1}
Let $E$ be minimal in $W\subset M^n\times \mathbb R$, and $x_0\in E\cap W.$ Let $r< \text{dist}(x_0, \partial W)$. Let $M = B_r(x_0)$ and $b, \rho_0$, $\alpha$, and $\bar R(M)$ be as in Theorem \ref{sict1} for $M^{n+1} = B_r(x_0)$ and $N^{n+1} = M\times\mathbb R$.We have
\begin{equation}\label{sic5}
|E \cap B(x_0, r)| \geq (\frac{r}{(n+1)c(n,\alpha)})^{n+1}
\end{equation}
\end{proposition}

\begin{proof}
We first claim that for any $B_\rho \subset\subset W$, we have that
\begin{equation}\label{sic6}
\int\limits_W |D\phi_E|\leq \int\limits_W |D\phi_{E\backslash B_\rho}|.
\end{equation}
From the minimality of $E$ in $W$ we get that
\begin{equation}\label{sic7}
\int\limits_{B_\rho}|D\phi_E| \leq \int\limits_{B_{\rho+\epsilon}} |D\phi_E|
\leq \int\limits_{B_{\rho+\epsilon}}|D\phi_{E\backslash B_\rho}|.
\end{equation}

Since $\phi_{E\backslash B_\rho}^+ = \phi_E^+$ on $\partial B_\rho$, $\phi_{E\backslash B_\rho}^- = 0$ on $\partial B_\rho$, and $\phi_{E\backslash B_\rho} \equiv 0$ on $B_\rho$, we have that
\begin{equation}\label{sic8}
\int\limits_{B_{\rho+\epsilon}}|D\phi_{E\backslash B_\rho}| =
\int\limits_{B_{\rho+\epsilon} \backslash \bar B_\rho}|D\phi_{E\backslash B_\rho}| + \int\limits_{\partial B_\rho}|\phi_E^+| dH_{n}.
\end{equation}

Since $B_{\rho + \epsilon}\backslash \bar B_\rho \to \emptyset$ as $\epsilon \to 0$, we get that
\begin{equation}\label{sic9}
\int\limits_{B_\rho}|D\phi_E| \leq \int\limits_{\partial B_\rho} |\phi_E^+| dH_{n}.
\end{equation}

Now, note that for every $\rho,$ we have from Lemma \ref{perl1} that
\begin{equation}\label{sic10}
\int\limits_W |D\phi_{E\cap B_\rho}| = \int\limits_{B_\rho}|D\phi_E| + \int\limits_{\partial B_\rho}|\phi_E^-| dH_n .
\end{equation}

Now, define $E_\rho \equiv E\cap B_\rho.$ From equation \eqref{sic9} and \eqref{sic10} we get that
\begin{equation}\label{sic11}
\int\limits_W |D\phi_{E_\rho}| \leq 2H_{n}(\partial B_\rho \cap E) = 2 \frac{d}{d\rho}|E_\rho|.
\end{equation}

By Theorem \ref{sict1}, we have that
\begin{equation}\frac{d}{d\rho}|E_\rho| \geq \frac{1}{c(n,\alpha)} |E_\rho|^\frac{n}{n+1}.\end{equation}

Therefore,
\begin{equation}\label{sic12}
|E_r| \geq (\frac{r}{(n+1)c(n,\alpha)})^{n+1}.
\end{equation}
\end{proof}


\section{Minimizers of Area}\label{moa}
\begin{theorem}\label{moat1}
Let $u\in BV_{loc}(\Omega)$, $\Omega \subset M^n$, minimize the area functional. Then u is locally bounded in $\Omega$.
\end{theorem}

\begin{proof}
Suppose that there exists a compact set $K \subset \Omega$ such that $u$ is not bounded on $K$. We may assume $K$ is a small closed geodesic ball. Let $R \leq \frac{1}{2}\Dist(K,\partial\Omega)$. Let $M = K \times I$ for a compact interval $I\subset\mathbb R$, and $N = \Omega \times \mathbb R$. Finally, let $b, \alpha, \bar R$, and  $\rho_0$ be as in Theorem \ref{sict1}. Note that the geometry of $\Omega \times \mathbb R$ gives that the results of Theorem \ref{sict1} only depend on the size of $I$.

 For every integer $m \geq 0$ there exists a point $x_m \in K$ such that $u(x_m) > 2mR$. Denote the subgraph of $u$ by $U = \{(x,y) | y< u(x)\}$. It follows that the points $z_i = (x_i, 2iR)\in U$.
 From Proposition \ref{sicp1} we have
\begin{equation}\label{moa1}
|U \cap B(z_i, R)| \geq C R^{n+1},
\end{equation}
where $c$ depends on the curvature of $\MRR$ and on $n$.

Then,
\begin{equation}\label{moa2}
\int\limits_{K_R}|u| \geq \sum\limits_{i=1}^m |U \cap B(z_i, R)| \geq cmR^{n+1},
\end{equation}
where $K_R \equiv \{x \in \Omega : \Dist (x,K) < R \}.$ Since $m$ is arbitrary,  this would imply that $\int\limits_{K_R}|u| = +\infty$, contrary to the hypothesis.
\end{proof}
\begin{remark}\label{moar1}
From this proof, we can see that, for a general boundary data $\varphi,$ in order to prove a local $C^0$ bound for the area functional minimizer $u$ we have to assume some restriction on $\bar R.$ For our later application (see Section \ref{tdp}), we can sacrifice the generality of boundary data $\varphi$
to get rid of the restriction on $\bar R.$
\end{remark}

\begin{theorem}\label{moat1'}
Assume $\partial\Omega$ is Lipschitz-continuous and let $\varphi\in L^{\infty}(\partial\Omega).$
Then $\mathcal A_\varphi(u, \mathfrak B),$ where $\Omega\subset\mathfrak B,$ attains its minimum at some $u$ in $BV(\Omega).$
Moreover, $u\in L^{\infty}(\Omega)$ and $\|u\|_{\infty}\leq M$ for some $M=M(\|\varphi\|_{\infty}).$
\end{theorem}
\begin{proof}
We define a constant subsolution and constant supersolution by
\begin{equation}\label{moa8}
\overline u=c_1>\|\varphi\|_{L^{\infty}(\mathfrak B)},
\end{equation}
and
\begin{equation}\label{moa9}
\underline{u}=c_2<-\|\varphi\|_{L^{\infty}(\mathfrak B)}.
\end{equation}
It's easy to see that both $\overline u$ and $\underline{u}$ satisfy equation
\begin{equation}\label{moa10}
\Div\frac{Dv}{W}=0,\,\,\mbox{in $\mathfrak B$}.
\end{equation}
Now, let $u_j\in BV(\Omega)$ be a minimizing sequence, that is
\begin{equation}\inf \mathcal A_\varphi(u, \mathfrak B)=\lim\limits_j\mathcal A_\varphi(u_j, \mathfrak B)=I.\end{equation}
Let us approximate the $u_j$ 's with smooth functions in $C^\infty(\Omega)$ which we still denote by $u_j$; see Theorem \ref{fbvt1}.

Set
\begin{equation}\overline u_j=\min\{u_j,\overline u\}\end{equation}
It's easy to verify using the second order nature of the area functional that
\begin{equation}\mathcal A_\varphi(u_j, \mathfrak B)\geq\mathcal A_\varphi(\overline u_j, \mathfrak B).\end{equation}
Analogously, set
\begin{equation}\underline u_j=\max\{\underline u, \overline u_j\},\end{equation}
we have
\begin{equation}\underline u\leq \underline u_j\leq\overline u.\end{equation}
Moreover,
\begin{equation}\label{moa11}
\mathcal A_\varphi(\overline u_j, \mathfrak B)\geq\mathcal A_\varphi(\underline u_j, \mathfrak B)
\end{equation}
Since $\underline u_j$'s are uniformly bounded in $BV(\mathfrak B),$ we can extract a
subsequence which converges in $L^1(\mathfrak B)$ to some function $u\in BV(\mathfrak B);$ see Theorem \ref{fbvt2}
Furthermore, $u\in L^{\infty}(\Omega)$ and $u=\varphi$ in $\mathfrak B-\bar\Omega.$
From the lower semicontinuity of the functional $\mathcal A_\phi$ we find that $u$ is the required minimizer.
\end{proof}


\section{Regularity}\label{reg}

\begin{lemma}\label{moal1}
On the set $L = \Omega - \Proj \Sigma \subset M^n$, where $\Sigma$ is the singular set of the subgraph $U$, the height function $u$ is regular.
\end{lemma}

\begin{proof}
It is well known that $H_{n-6}(\Sigma)=0$ (see Simon\cite{Sim}).
To see that $u$ is regular on $L$, it is sufficient to show that $\nu_{n+1} > 0$ on $\partial U \backslash \Sigma$. Suppose on the contrary that at a point $x_0 \in \partial U \backslash \Sigma$, we have $\nu_{n+1} = 0$. Then, in a neighborhood $V$ of  $x_0$ we have $\nu_{n+1}(x) \geq 0$.
Note that,
\begin{equation}\label{moa3}
\triangle_S \nu_{n+1} + (|A|^2 + \Ric(N)) = 0.
\end{equation}
 Let $C^+ = (|A|^2 + \Ric(N))^+$ and $C^- = (|A|^2 + \Ric(N))^-$. We have that
\begin{equation}\label{moa4}
\triangle_S \nu_{n+1} + C^- \nu_{n+1} = - C^+ \nu_{n+1} \leq 0,
\end{equation}
and so $\nu_{n+1} \equiv 0$.

Therefore, $\nu_{n+1}$ vanishes identically in a neighborhood of $V$ of $x_0$. Let $\Gamma = \Proj V$. We have $H_{n-1}(\Gamma) > 0$. If $z \in \Gamma$, the vertical straight line through $z$ contains a point $x \in \partial U \backslash \Sigma$ with $\nu_{n+1}(x) = 0$. If the line does not meet $\Sigma$ then our above argument gives that $v_{n+1}$ on an open set of the line. From the connectedness of $\mathbb R$, we then have that if the line does not meet $\Sigma$, then it lies entirely on $\partial U$. This is impossible since $u$ is locally bounded. Therefore, we must have $\Gamma \subset \Proj \Sigma,$ but then $H_{n-1}(\Sigma) > 0$ which is a contradiciton.
\end{proof}

\begin{proposition}\label{moap1}
Let $u \in BV_{loc}(\Omega)$ minimize the area in $\Omega$. Then $u \in W^{1,1}_{loc}(\Omega).$
\end{proposition}

\begin{proof}
Consider $\Sigma$, the singular set of the subgraph $U$. Let $S = \Proj(\Sigma)$. We have seen that $u$ is regular in $\Omega \backslash S$ and $H_{n-6}(S) = 0$. In particular $|S| = 0$. If $A \subset \subset \Omega$ is an open set, we have
\begin{equation}\label{moa5}
\int\limits_A \sqrt{1 + |Du|^2} = \int\limits_{A\backslash S} \sqrt{1 + |Du|^2} + \int\limits_{S\cap A} \sqrt{1+|Du|^2}.
\end{equation}
On the other hand, Lemma \ref{perl3} tells us that $P(U,S \times \mathbb R) = 0$, and therefore
\begin{equation}\label{moa6}
\int\limits_A \sqrt{1+|Du|^2} = \int\limits_{A\backslash S}\sqrt{1 + |Du|^2}.
\end{equation}
Hence $\int\limits_{S\cap A}|Du| \leq \int\limits_{S\cap A} \sqrt{1+|Du|^2} = 0$, and so $u\in W^{1,1}_{loc}(\Omega)$.

\end{proof}

Note that the area functional is not strictly convex on $BV(\Omega)$. For example, on compact sets of $\mathbb R$ consider different translations of the Heavyside function. However, since we now know that our minimizer $u \in W^{1,2}(\Omega)$, we have a uniqueness result for minimizers of the functional $J(u,\Omega)$ \cite{Giu}.

\begin{proposition}\label{moap2}
Let $\Omega$ be connected and let $\phi \in L^1(\partial \Omega)$. Suppose $u,v$ are two minimas of the functional $J(u,\Omega)$. We have that
\begin{equation}\label{moa7}
 v = u + \text{constant}.
\end{equation}
\end{proposition}


Following Giusti\cite{Giu}, we then have a regularity theorem for minimizers $u$.

\begin{theorem}\label{regp1}
Let $u \in BV_{loc}(\Omega)$ minimize locally the functional $\intO\sqrt{1 + |Du|^2}$. Then, u is Lipschitz-continuous
(and hence analytic) in $\Omega$.
\end{theorem}

\begin{proof}
Let $\mathcal B\equiv \mathcal B(x_0,R)$ be a small ball in $\Omega$. We have
\begin{equation}\label{reg1}
\int\limits_{\mathcal B} \sqrt{1+|Du|^2} dx \leq  \int\limits_{\mathcal B}\sqrt{1+|Dw|^2} dx + \int\limits_{\partial \mathcal B} |w - u| dH_{n-1}
\end{equation}
for every $w \in BV(\mathcal B)$. Since the singular set $S$ satisfies $H_{n-6}(S) = 0$, we can find a descending sequence of open sets $S_n$ such that $S_{n+1} \subset \subset S_n$, $\bigcap \limits_n S_n = S$, and $H_{n-1}(S_j \cap \partial \mathcal B) \to 0$.

Now, let $\phi_j$ be a smooth function  on $\partial \mathcal B$ satisfying
\begin{align*}
  &\phi_j   =  u \text{ in } \partial\mathcal B \backslash S_j\\
   &\sup\limits_{\partial\mathcal B}|\phi_j|\leq 2 \sup \limits_{\partial \mathcal B}|u|,   \\
  & \phi_j \to u \text{ in } L^1(\partial \mathcal B).
\end{align*}

It is well known that there exists a unique solution $u_j$ of the Dirichlet Problem with boundary datum $\phi_j$ on $\partial \mathcal B$ (see \cite{GT}). The functions $u_j$ are smooth in $\mathcal B$, and moreover
$\sup\limits_{\mathcal B}|u_j| \leq 2\sup\limits_{\partial \mathcal B} |u|.$ We have
\begin{equation}\label{reg2}
\int\limits_{\mathcal B}\sqrt{1+|Du_j|^2} \leq \int\limits_{\mathcal B}\sqrt{1+|Dw|^2} + \int\limits_{\mathcal B}|w-\phi_j|dH_{n-1}
\end{equation}
for every $w \in BV(\mathcal B)$.

From the a-priori estimate of the gradient (see Theorem 1.1 in \cite{Spr}), we conclude that the gradients $Du_j$ are equibounded
in every compact set $K \subset \mathcal B$. Using the Arzela-Ascoli Theorem and passing to a subsequence, we get uniform convergence on compact subsets of $\mathcal B$ to a locally Lipschitz-continuous function $v$. Taking $w = 0$, in equation \eqref{reg2}, we get
\begin{equation}\label{reg3}
\int\limits_{\mathcal B} \sqrt{1+|Du_j|^2} \leq |\mathcal B| +  \int\limits_{\partial \mathcal B} |\phi_j|dH_{n-1} \leq C.
\end{equation}
Therefore, $v \in BV(\mathcal B)$. Furthermore, \eqref{reg2} and the lower semi-continuity of the area functional give us that $v$ minimizes $J(u,\phi)$ on $\mathcal B$. Therefore, Proposition \ref{moap1} gives us that $v \in W^{1,1}(\mathcal B).$

We want to prove that $v$ has trace $u$ on $\partial \mathcal B$. For that, let $y \in \partial \mathcal B$ be a regular point for $u$. For $j$ sufficiently large, $y \in \partial\mathcal B \backslash S_j$, and therefore for all $k> j$, $\phi_k = u$ in a neighborhood of $y$ in $\partial\mathcal B$. We can therefore construct two functions $\phi^+$ and $\phi^-$, both of class $C^2$ on $\partial\mathcal B$ such that \\
(i) $\phi^\pm = u $ in a neighborhood of $y$ in $\partial\mathcal B,$ \\
(ii) $\phi^- \leq \phi_k \leq \phi^+$ in $\partial\mathcal B$ for $k > j.$\\
Let $u^{\pm}$ be the solutions of the Dirichlet problems with boundary data $\phi ^ \pm$ respectively. We have
$ u^- \leq u_k \leq u^+$ for all $k>j$. Hence we get that
\begin{equation}\label{reg4}
u^- \leq v \leq u^+.
\end{equation}
Therefore, $v = u$ at every regular point $y \in \partial \mathcal B$. Since $H_{n-1}(S) = 0$
we have that $v$ has trace $u$ on $\partial \mathcal B$. So equation \eqref{reg2} gives us that
\begin{equation}\label{reg5}
\int\limits_{\mathcal B}\sqrt{1+|Dv|^2} \leq \int\limits_{\mathcal B}\sqrt{1 + |Dw|^2} + \int\limits_{\partial\mathcal B}|w-u|dH_{n-1}.
\end{equation}

Since $v = u$ on $\partial \mathcal B,$ by Proposition \ref{moap2}, we have that $u = v$.
This implies that $u$ is Lipschitz-continuous and hence, analytic in $\Omega$.

\end{proof}


\section{The Dirichlet Problem}\label{tdp}

One application of Theorem \ref{funt1} is to study the solvability of
the Dirichlet problem:
\begin{align}\label{tdp1}
\Div \frac{Du}{W} & =0  \text{    in } \Omega\subset M^n,\\
u & =\varphi   \text{    on } \partial \Omega .
\end{align}
Here, we are going to follow the argument in Schulz-Williams\cite{SW} to construct
barriers on a neighborhood of the boundary, and we will show that the general solution $u$
(as obtained before) is the solution of \eqref{tdp1} when $\varphi$ satisfies certain conditions.

\begin{proposition}\label{tdpt1}
Given $n\geq 2,$ $K\in(0,1/\sqrt{(n-1)\gamma}),$ and $\gamma>1,$ there exists $\epsilon>0$
depending on $n, K, M^n,$ and $\partial\Omega$ such that if
\begin{enumerate}
\item $x^0\in\partial\Omega$ and $\partial\Omega$ is $C^2$ near $x^0$;
\item $\varphi\in L^1(\partial\Omega)$ satisfies
\begin{equation}\label{tdp2}
\varphi\leq\varphi(x^0)+\min\{K\|x-x^0\|_{M^n}, \epsilon\}
\end{equation}
where $x\in\partial\Omega\cap\mathcal N_0,$ $\mathcal N_0$ is some neighborhood of $x^0$;
\item u is the generalized solution of the Dirichlet problem
\end{enumerate}
then
\begin{equation}\label{tdp3}
\lim\limits_{x\rightarrow x_0}\sup\limits_{x\in\Omega}u(x)\leq\varphi(x^0).\\
\end{equation}
Furthermore, there is a constant $C$ depending on $n,$ $K,$ $M^n,$ and $\partial\Omega$
such that
\begin{equation}\label{tdp4}
u(x)\leq\varphi(x^0)+C\|x-x^0\|_{M^n},\,x\in\Omega\cap\mathcal N_0.
\end{equation}
\end{proposition}

\begin{remark}\label{tdpr1}
Condition (1) can be relaxed to an exterior ball condition on $\partial\Omega$ at $x^0,$
but for convenience we assume $\partial\Omega$ is $C^2$ near $x^0.$
\end{remark}

\begin{proof}
Since $\partial\Omega$ is $C^2$ near $x^0,$ we may assume $x^0=0$ and the interior unit normal to
$\partial\Omega$ at $x^0$ is $e_n.$ Near $x^0,$ we have local geodesic normal coordinates $x_1,\cdots, x_n$ for $M^n,$
and we may assume $\partial\Omega$ is given by
\begin{equation}\{(x', w(x'))| x'=(x_1,\cdots,x_{n-1})\}\end{equation}
near $x^0.$ Here $w$ is a $C^2$ function with $w(0)=0,\,Dw(0)=0$ and $|D^2w(0)|\leq L.$ Note that $d(x^0, x) = \sum_i x^i x^i$. For the duration of this proof, subscripts will denote derivatives respect to the chosen coordinate system.

Now consider the function
\begin{equation}\label{tdp5}
\psi(x):=K^2\sum\limits_{i=1}^{n} x_ix_i+2\alpha(x_n-w(x')),
\end{equation}
where $\alpha$ is a constant to be chosen later. The metric on $M^n$ will be denoted by $\sigma_{ij}$.
Let $v(x):=\psi^{1\backslash 2},$ then we have

\begin{align}\label{tdp6}
v_i& =\frac{1}{2}\psi^{-1\backslash 2}\psi_i=\frac{1}{2v}\psi_i \\
\label{tdp7}
v_{ij}& =-\frac{1}{4}\psi^{-3\backslash 2}\psi_i\psi_j+\frac{1}{2}\psi^{-1\backslash 2}\psi_{ij}.
\end{align}
Moreover,
\begin{align}\label{tdp8}
	\psi_i& =2K^2 x_i-2\alpha w_i,& \,1\leq i\leq n-1, \\
\label{tdp9}
	\psi_n& =2K^2 x_n+2\alpha,& \\
\label{tdp10}
	\psi_{ij}& =2K^2\delta_{ij}-2\alpha w_{ij},& \,1\leq i\leq j\leq n.
\end{align}

Now denote $Q=g^{ij}D^2_{ij}$ where we are using the connection $D$ on $M^n$ and $g^{ij} = \sigma^{ij} - \frac{v^iv^j}{1+|\nabla v|^2}$ is the inverse of the metric on the graph of $v$. We are using $\sigma$ to raise and lower indices, so $v^i =\sigma^{ij}v_j$. We have
\begin{equation}\label{tdp11}
\begin{aligned}
Qv&=g^{ij}v_{ij} - g^{ij}\Gamma_{ij}^k v_k\\
&=g^{ij}\{-\frac{1}{4}v^{-3}\psi_i\psi_j+\frac{1}{2}v^{-1}\psi_{ij}\}- g^{ij}\Gamma_{ij}^k v_k\\
&=-\frac{1}{v}\{g^{ij}v_iv_j-K^2g^{ii}+\alpha g^{ij}w_{ij}\}- g^{ij}\Gamma_{ij}^k v_k \\
&=-\frac{1}{v}\{g^{ij}v_iv_j-K^2g^{ii}+\alpha g^{ij}w_{ij}+ K^2 \Gamma^l_{ij} g^{ij} x_l - \alpha \Gamma^l_{ij} g^{ij} w_l + \alpha \Gamma^n_{ij}g^{ij}\}.
\end{aligned}
\end{equation}
 For brevity of notation, define $W=\sqrt{1+|\nabla v|^2}.$ Since,
\begin{equation}\label{tdp12}
g^{ij}v_iv_j=\frac{|\nabla v|^2}{1+|\nabla v|^2}
\end{equation}
and
\begin{equation}\label{tdp13}
g^{ii}=\sigma^{ii} - \frac{v^iv^i}{W^2},
\end{equation}
we have
\begin{align}\label{tdp14}
Qv=-\frac{1}{v}\left\{\frac{W^2-1}{W^2}+K^2\left[\frac{v^iv^i}{W^2} -\sigma^{ii}\right]+\alpha g^{ij}w_{ij}+ K^2 \Gamma^l_{ij} g^{ij} x_l - \alpha \Gamma^l_{ij} g^{ij} w_l + \alpha \Gamma^n_{ij}g^{ij}\right\}.
\end{align}
Note that as $x \to 0$ and for fixed $\alpha > 0$ we have that $|v_i| = (K + 1/K) O(1) $ for $1\leq i \leq n-1$. We also have that 
\begin{equation}
|v_n| = K O(1) + \frac{\alpha}{\sqrt{K^2 \sum_{i=1}^n x_i^2 + 2\alpha(x_n - w(x'))}}.
\end{equation}
Therefore, as $x \to 0$, we have $W \to \infty$, $\frac{v^iv^j}{W^2} \to \frac{\sigma^{in}\sigma^{jn}}{\sigma^{nn}} = \delta_{in}\delta_{jn}$,
and $g^{ij} \to \sigma^{ij} - \frac{\sigma^{in}\sigma^{jn}}{\sigma^{nn}} = \delta_{ij} - \delta_{in}\delta_{jn}$. Hence, as $x\to x_0$ we have
\begin{equation}
Qv \to  -\frac{1}{v}(1 + K^2(1 - n) + \alpha \sigma^{ij} w_{ij} ).
\end{equation}

So, by first choosing $\alpha>0$ small enough depending only on $K,n$, and $L$, we may then choose a neighborhood of $x_0$ where $Q v < 0$. Note that from our assumptions on $\varphi$ we get that $v \geq \varphi$ on $\partial \Omega$. Therefore, $v$ is a supersolution, and \eqref{tdp3}, \eqref{tdp4} follows.

\end{proof}

As an application we have the following theorem.

\begin{theorem}\label{tdpc1}
Let $\Omega$ be a bounded open subset of $M^n$ with $C^2$ boundary $\partial\Omega.$
Given $n\geq 2,$ $K\in\left(0, \frac{1}{\sqrt{(n-1)\gamma}}\right),$ and $\gamma>1,$ there exists an $\epsilon > 0$ depending only on $K$, $\Omega$, and $\partial\Omega$ such that the following holds.
Suppose $\varphi$ is Lipschitz continuous on $\partial\Omega$ and
\begin{equation}|\varphi(x)-\varphi(y)|\leq K\|x-y\|_{M^n},\,\,\mbox{if $x, y\in\partial\Omega$},\end{equation}
\begin{equation}\sup\limits_{\partial\Omega}\varphi(x)-\inf\limits_{\partial\Omega}\varphi(x)\leq\epsilon.\end{equation}
Then there exists a function $u\in C^2(\Omega)\cap C^0(\bar\Omega)$ such that u is the solution of the equation \eqref{tdp1}.
\end{theorem}


\section{Mean Curvature Flow}\label{mcf}

During this section, for a function $u: M\to \mathbb R$, we will often need to make distinction between the covariant derivative $Du$ of $u$ as a function on M and the covariant derivative $\nabla u$ of $u$ as a function on the hyper-surface that is the graph of $u$. We will also need to make distinction between the metric $\sigma_{ij}$ on $M$ and the metric $g_{ij}$ on the graph of $u$. Since our calculations take place on $M$, we will use the convention that we only use $\sigma_{ij}$ and its inverse $\sigma^{ij}$ to raise and lower indices. For convenience, we will often use $W \equiv \sqrt{1+|Du|^2}$.

Note that the upwards normal to the graph of $u$ is $N = (1/W) (-Du, 1)$. Furthermore, by extending any function $g: M \to \mathbb R$ to $\bar g: M\times\mathbb R \to \mathbb R$ using $\bar u(x,t) = u(x)$, we easily see that
\begin{equation} \label{eq:9_1}
|\nabla g|^2 = |Dg|^2 - \la Dg, N \ra^2 \geq |Dg|^2(1 - |Du|^2/W^2) = |Dg|^2 / W^2.
\end{equation}

Now, we consider the problem of the graphical mean curvature flow
\begin{equation}\label{eq:mcf}
\begin{cases}
u_t(x,t)  =  nWH(u) = \triangle_M u - \frac{u^i u^j}{W^2} D^2_{ij} u & \text{in } \Omega\times(0,\infty), \\
u(x,t) = \phi(x) & \text{on } \partial \Omega \times (0,\infty), \\
u(x,0)  = u_0(x), &
\end{cases}
\end{equation}
where $g^{ij} = \sigma^{ij} - \frac{u^iu^j}{W^2}$. Here, we assume no conditions on the mean curvature of $\partial \Omega$. We only assume that all data are $C^\infty$ smooth, and that our compatibility condition,
\begin{equation}
u_0 = \phi \text{ on } \partial \Omega,
\end{equation}
is of order zero.

Like Oliker and Ural'tseva \cite{OU1993}\cite{OU1997}, we use solutions to the perturbed and regularized problem
\begin{equation}\label{eq:emcf}
\begin{cases}
\ue_t(x,t) = \triangle_M \ue - \frac{u^{\epsilon i} u^{\epsilon j}}{W^2}D^2_{ij}\ue - \epsilon \sqrt{1+|D \ue|^2_M} \triangle_M \ue & \text{in } \Omega\times(0,\infty), \\
\ue(x,t) = \phi(x) & \text{on } \partial\Omega \times (0,\infty) ,\\
\ue(x,0) = u_0(x).
\end{cases}
\end{equation}

For convenience of notation, we will introduce the operator $L^\epsilon f= \triangle_M f- \frac{f^if^j}{1+|Df|^2} - \epsilon \sqrt{1+|Df|^2} \triangle_M f$. Note that this regularized problem is uniformly parabolic for every $\epsilon > 0$, but we only have the zeroth order compatibility condition $u_0 = \phi$ on $\partial \Omega \times \{0\}$. Therefore, we are guaranteed to have the existence of a unique solution $\ue \in C^\infty(\Omega\times [0, \infty)) \cap C^\infty(\partial \Omega \times (0,\infty)) \cap C^0(\bar \Omega \times [0,\infty))$ \cite{Lie}. That is, $\ue$ is continuous everywhere but only $C^\infty$ away from the edge of $\Omega \times [0,\infty)$. 

By establishing appropriate estimates (uniform in $\epsilon$) for the perturbed problem \eqref{eq:emcf}, we show that there are sequences $\epsilon_i \to 0$ where we get uniform convergence in $C^\infty$ to a function $u\in C^\infty(\Omega\times(0,\infty)) \cup L^\infty([0,\infty); W^{1,1}(\Omega))$ on compact subsets of $\Omega\times (0,\infty)$. Clearly, $u$ mast satisfy the equation of the mean curvature flow. Furthermore, there is a function $\bar u$ depending only on the sequence $\epsilon_i$ such that there are sequences $t_j$ such that $u(\cdot, t_j)$ converges uniformly in $C^\infty$ on compact sets of $\Omega$ to $\bar u$. We have that $\bar u$ is independent of the choice of sequence $t_j$ for a fixed sequence $\epsilon_i$.

Just like Oliker and Ural'tseva \cite{OU1993}, we have the following estimates. Our proof is similar to that of Oliker and Ural'tseva\cite{OU1993}, but we include it for completeness.

\begin{theorem}

The solution $\ue$ to \eqref{eq:emcf} satisfies the following estimates (uniform in $\epsilon$):
\begin{align}
\sup\limits_{\bar\Omega\times [0,\infty)} |\ue| & \leq  C, \label{eq:ubound}\\
\sup\limits_{\Omega\times (0,\infty)} |\ue_t| & \leq C, \label{eq:utbound}\\
\sup\limits_{t\in(0,\infty)} \int\limits_\Omega \sqrt{1+|D\ue|^2} + \epsilon |D\ue|^2  \dx & \leq C \label{eq:intdubound}.
\end{align}
Here $C = C(u_0,\phi, \Omega)$.

\end{theorem}

\begin{proof}

Inequality \eqref{eq:ubound} follows easily from the maximum principle, but inequality \eqref{eq:utbound} does not. We do not have any guarantee that $u_t$ is continuous on the edge of the parabolic domain. In order to deal with this, we introduce a regularization of \eqref{eq:emcf} that has a first order compatibility condition by changing the boundary conditions on $\partial \Omega \times [0,\infty)$. Let $\psi(t) \in C^\infty([0,\infty))$ be such that $\psi(0) = 0$, $\suppt \psi \subset [0,2]$, $\psi'(0) = 1$, and $|\psi'(0)| \leq 1$. Consider the problem
\begin{equation} \label{eq:edmcf}
 \begin{cases}
\ued_t(x,t) = L^\epsilon \ued & \text{in } \Omega\times(0,\infty), \\
\ued(x,t) = \phi(x) + \delta \psi(t/\delta) L^\epsilon u_0& \text{on } \partial\Omega \times (0,\infty), \\
\ued(x,0) = u_0(x).
\end{cases}
\end{equation}

The problem \eqref{eq:edmcf} is uniformly parabolic and satisfies a first order compatibility condition. So, we are guaranteed to have a solution $\ued$ with $\ued_t$ continuous. Since \eqref{eq:edmcf} has no zeroth order terms for $\ued$, we find that $\ued_t$ satisfies an equation of the form
\begin{equation} \label{eq:utflow}
\ued_{tt} = a^{ij}(x,t) D^2_{ij} \ued_t + b^i(x,t) D_i \ued_t,
\end{equation}
where $a^{ij}, b^i$ are smooth on $\Omega\times(0,\infty)$. On $\partial \Omega\times [0,\infty)$ we have that $|\ued_t| \leq \sup |\psi'(t/\delta) L^\epsilon u_0(x)| \leq \sup |L^\epsilon u_0(x)|.$ Therefore, by the maximum principle, we have that $|\ued_t| \leq \sup |L^\epsilon u_0|$.

Now, we may use that $\ue - \ued$ also satisfies an equation of the form \eqref{eq:utflow}. Therefore, from the boundary conditions of \eqref{eq:edmcf} and $\ue, \ued \in C(\bar \Omega \times [0,\infty))$, we see that $| \ue - \ued | \leq \delta \sup\limits_{\partial \Omega} |L^\epsilon u_0|$. Therefore, $\ued \to \ue$ uniformly as $\delta \to 0$. So, therefore, for any $(x,t) \in \Omega \times (0,\infty)$, we have that
\[
|\frac{\ue(x,t+h) - \ue(x,t)}{h}| \leq \sup |L^\epsilon u_0|.
\]
Therefore, away from the edge of the domain, we have that $|\ue_t| \leq |L^\epsilon u_0| \leq C$. Hence, we have \eqref{eq:utbound}.

To show \eqref{eq:intdubound}, we use a cutoff function $\eta \in C_0^1 (\Omega)$, the fact that $\frac{\ue_t}{W} - \Div_M \frac{D\ue}{W} - \epsilon \triangle_M \ue = 0$ in $\Omega \times (0,\infty)$, and integration by parts to get
\begin{equation} \label{eq:9_9}
\int\limits_{\Omega\times\{t\}} \frac{\ue_t}{W}\eta + \frac{\la D\ue, D\eta \ra}{W} + \epsilon \la D\ue, D\eta \ra = 0.
\end{equation}
We then use the choice of test function $\eta(x) = \ue(x,t) - u_0(x)$ in \eqref{eq:9_9} to get
\begin{equation}
\int\limits_{\Omega\times\{t\}} \sqrt{1+|D\ue|^2} + \epsilon |D\ue|^2 = \int\limits_{\Omega\times\{t\}} \frac{1+ \la D\ue, Du_0 \ra - (\ue - u_0)\ue_t}{W} + \epsilon \la D\ue, D u_0 \ra.
\end{equation}
Using a Cauchy-Schwarz inequality we then have
\begin{align}
\int\limits_{\Omega\times\{t\}} \sqrt{1+|D\ue|^2} + (\epsilon/2) |D\ue|^2 & \leq  \int\limits_{\Omega\times\{t\}} \frac{1+ \la D\ue, Du_0 \ra - (\ue - u_0)\ue_t}{W} + (\epsilon/2) |Du_0|^2 \\
& \leq C.
\end{align}
From this, \eqref{eq:intdubound} follows.
\end{proof}

Now, to guarantee the convergence of $\ue$ as $\epsilon \to 0$ we need uniform (in $\epsilon$) estimates on the spatial derivatives $|D\ue|$. Oliker and Ural'tseva \cite{OU1993} use an iteration scheme to construct estimates for $|D\ue|$. This iteration scheme depends on the use of the Sobolev inequality. First they make use of iteration and the Sobolev inequality $|u|_{L^{p^*}(\mathbb R^n)} \leq C_n (\int\limits_{\mathbb R^n} |Du|^p)^{1/p}$ for functions on $\mathbb R^n$ to put estimates on $\epsilon |D\ue|.$

From this estimate, they use a form of the Sobolev inequality by Ladyzhenskaya-Ural'tseva \cite{LU} for functions on graphs realized as surfaces that obey elliptic equations of a certain type. Letting $S_u$ being the graph of u, this Sobolev Inequality takes the form $\int\limits_{S_u} f^2 \leq D H_n(S_u)^{2/n} \int\limits_{S_u} |\nabla f|^2$ where $D$ depends on certain estimates of the coefficients of the elliptic equation. From here, they are able to obtain uniform bounds for $|D\ue|$.

We will use a similar procedure, but we must be careful to use our quantities in a tensorial manner. Following Oliker and Ural'tseva \cite{OU1993}, on the tangent bundle $TM$ we define the quantities
\begin{equation}
F : TM \to \mathbb R
\end{equation}
\begin{equation}
F^\epsilon : TM \to \mathbb R
\end{equation}
by
\begin{align}
F(x,v) & = \sqrt{1 + |v|_M^2}, \\
F^\epsilon(x,v) & = F(x,v) + (\epsilon/2) |v|_M^2.
\end{align}

We will need to make use of some geometry on $TM$. A good reference for some of this discussion in contained in Gudmundsson-Kappos\cite{GK} and doCarmo\cite{Car}. Let $\pi: TM \to M$ be the natural bundle map for the tangent bundle. Consider any vector $X \in T_p M$. There is a natural map sending $X$ to the vertical fiber which we will denote by $X_v$. In canonical coordinates $(x,w)$, we have that if $X = X^i \partial_i$ then $X_v^i =  (0, X^i)$. The Levi-Cevita connection on $M$ gives us a way to identify a horizontal vector space of $T_{(p,w)}TM$. A vector $X\in T_pM$ lifts to a vector $X_h\in T_{(p,w)} TM$. $X_h$ is determined as follows. Let $p(t)$ be a curve in $M$ such that $p(0) = p$ and $p'(0) = X$. Let $(p(t), w(t))$ be the curve in $TM$ determined by the parallel transport of $w$ in the direction of $X$. Then, $X_h = (X, w'(t))$. Hence, we see that in canonical coordinates $X^h (p,w) = (X, -\Gamma^k_{ij} w^i X^j)$. Every tangent space $T_{(p,w)} TM$ can be decomposed into a direct sum of horizontal and vertical spaces.

For any frame $\{ e_i \}$, we denote the horizontal lift as $\{e^h_i\}$ and the vertical lift as $\{e^v_i\}$. The Sasaki metric $\bar \sigma$ on $TM$ is defined by \cite{GK}, \cite{Kow}
\begin{enumerate}
\item  $\bar\sigma (X_h, Y_h) = \sigma(X,Y)$,
\item $\bar\sigma(X_h, Y_v) = 0$,
\item $\bar\sigma(X_v, Y_v) = \sigma(X,Y).$
\end{enumerate}
From the formula for $X_h$ and $X_v$ we have the following. For $i = 1,2$, let $(\gamma_i(t), w_i(t))$ be a curve in $TM$ with $(\gamma_i(0), w_i(0)) = (p,w)$ and $(\gamma_i'(0), w_i'(0)) = X_{hi} + Y_{vi}$. Then $\bar \sigma (X_{h1} + Y_{v1}, X_{h2} + Y_{v2}) = \sigma(X_1, X_2) + \sigma\left(\frac{D}{dt} w_1(0), \frac{D}{dt} w_2(0)\right)$.
Finally, we have the following formulas for the Levi-Cevita connection $\bar D$ associated with the Sasaki metric \cite{GK}, \cite{Kow}:
\begin{align}
(\bar D_{X_h} Y_h)_{(p,w)} & = (D_X Y)_{h,(p,w)} - \frac{1}{2} \left(R_p(X,Y)w\right)_v,\label{eq:conneq1}\\
(\bar D_{X_h} Y_v)_{(p,w)} & = (D_X Y)_{v,(p,w)} + \frac{1}{2} \left(R_p(w,Y)X\right)_h,\label{eq:conneq2}\\
(\bar D_{X_v} Y_h)_{(p,w)} & = \frac{1}{2} \left( R_p(w, X)Y\right)_h,\label{eq:conneq3}\\
(\bar D_{X_v} Y_v)_{(p,w)} & = 0\label{eq:conneq4}.
\end{align}

We will also need to make use of a certain type of pull-back of tensors on $TM$. Note that the map $p\in M \to (p, D\ue) \in TM$ gives us a section of $TM$. We use this to construct a sort of identification map $f : T_p M \to T_{(p,D\ue)} TM$ by $(p,v) \to (p,D\ue,0, v)$. Note that $f_p(W) = W_v(p,D\ue)$, the vertical lift of $W$ at $(p,D\ue)\in TM$. If $T$ is a tensor on $TM$, then we use $f$ to define a pull-back by $f^* T_p(X) = T_{(p,D\ue)}(f(X))$. For brevity and clarity of notation we will use the notation $T_{\is} = f^*(T)_i$ and $T_{\js,\is} = f^*(D_i T_j)$. For scalar functions $G: TM \to \mathbb R$, we will also use the shorthand notation $G_{\is} = f^*(\bar DG)_i$ and $G_{\is\js} = f^*(\bar D^2 G)_{ij}$. We will see later that $F_{\is}$ is the geometric equivalent of the euclidean case of $\frac{\partial}{\partial w_i} \sqrt{1+|w|^2}$. We have a similar result for $\fe_{\is}$. In this section, we will often use subscripts to denote the appropriate covariant derivatives of scalar functions.

Now, we make some remarks about the differential of the map $g(x): M \to TM$ given by $g(x) = (x, D\ue)$. Consider canonical coordinates $(x_i, w^j)$. We have $g(x) = (x_i, u^{\epsilon j})$, therefore $dg(\partial_l) = (\partial_l, \partial_l (u^{\epsilon j}) \partial_j) = (\partial_l, -\Gamma^j_{lk} u^{\epsilon k}\partial_j) + (D_l D \ue)_v = (\partial_l)_h + (D_lD \ue)_v.$ So, it is clear that 
\begin{equation}\label{eq:9_22_0}
dg(X) = X_h + (D_X D \ue)_v.
\end{equation}

First, we write down a simple computational lemma that will be needed.

\begin{lemma}
If $G: TM \to \mathbb R$ is a function independent of the horizontal directions on $TM$ (i.e. $\bar D_{X_h} G = 0$), then
\begin{equation}\label{eq:9_22}
D_i \bar D_{\js} G = (D^2_{ik} \ue)  \bar D^2_{k^* \js}G = (D^2_{ik}\ue)G_{\js \ks}.
\end{equation}
\end{lemma}

\begin{proof}

Let $\{e_i\}$ be a geodesic orthonormal frame on $M$ such that $De_i (p) = 0$. Consider the orthonormal frame $\{e_i^h, e_i^v\}$ on $TM$. We have that
\begin{equation}
\bar D_{\js} G= (f^* \bar D G)_j = \bar D_{e^v_j}G \circ (x,D \ue).
\end{equation}
Since we are using a geodesic frame, we have that
\begin{equation}
D_i \bar D_{\js} G(p) = D_i (\bar D_{\js} G).
\end{equation}
So, from \eqref{eq:9_22_0} and the chain rule we have
\begin{equation}\label{eq:9_25}
D_i\bar D_{\js} G(p) = \bar D_{e_i^h}(\bar D_{e^v_j} G) + (D^2_{ik}\ue)\bar D_{e_k^v}(\bar D_{e^v_j}G).
\end{equation}
From a Leibniz formula, $dG(X_h) = 0$, and (\ref{eq:conneq2}), we get that
\begin{align}
\bar D_{e_i^h}(\bar D_{e^v_j}G) & = \bar D^2_{e_i^h e^v_j}G + dG(\bar D_{e_i^h}e^v_j), \\
& = \bar D^2_{e_i^h e^v_j}G.
\end{align}
Then, we use the symmetry of the Hessian, the definition of the second covariant derivative, and \eqref{eq:conneq3} to get
\begin{equation}
= \bar D_{e_j^v}(\bar D_{e_i^h} G) = 0.
\end{equation}
Similarly, we apply (\ref{eq:conneq4}) to the second term of \eqref{eq:9_25} to get
\begin{equation}
D_i \bar D_{\js}G(p) =  (D^2_{ik}\ue)\bar D^2_{e_k^v e^v_j} G.
\end{equation}
This gives the lemma.
\end{proof}

Consider the function $h: TM \to \mathbb R$ given by $h(p,w) = |w|^2$. Now, in canonical coordinates,
\begin{align}
\bar D_{(\partial_i)_h} h & = \bar D_{(\partial_i)_h} |w|^2, \\
& = \frac{\partial}{\partial x_i} |w|^2 - \Gamma^k_{ij}w^j \frac{\partial}{\partial w^k} |w|^2, \\
& = 2\sigma(w, D_{\partial_i} w) - 2 \Gamma^k_{ij}w^j g_{kl}w^l \\
& = 0.
\end{align}
From this, we can see that $\bar D_{X_h} F = \bar D_{X_h} \fe = 0$. Finally, using canonical coordinates, it is clear that
\begin{align}
\bar D_{(\partial_i)_v} F & = \frac{\partial}{\partial w^i} \sqrt{1 + |w|^2} = \frac{g_{ij}w^j}{\sqrt{1+|w|^2}}, \\
& = \frac{\sigma(\partial_i, w)}{\sqrt{1+|w|^2}}
\end{align}
So, we see that $F_{\is} = \bar D_{\is} F = \bar D_{e_i^v} F = \frac{D_i \ue}{W}.$ Similarly, $\fe_{\is} = \frac{D_i \ue}{W} + \epsilon D_i \ue. $

Now, we follow the calculations of Oliker-Ural'tseva\cite{OU1993} for our similar tensors coming from $F$ and $\fe$. Using a geodesic frame $\{e_i\}$ on $M$ it is clear that $D_i \fe_{\is} = \frac{\triangle \ue}{W} - \frac{D^2_{ij}\ue D_i \ue D_j \ue}{W^2} + \epsilon \triangle \ue$. We immediately see that \eqref{eq:emcf} becomes
\begin{equation} \label{eq:9_19}
\frac{\ue_t}{W} - D_i \fe_{\is} = 0.
\end{equation}
Using \eqref{eq:9_22}, we note that $(F_{\js}) D_j\fe_{\is} = W_k \fe_{\ks\is}$. Also, note that $F_{\is} D_i \ue_t = W_t$. Rewriting \eqref{eq:9_19} as $\ue_t - W D_i \fe_{\is} = 0$, we apply the operator $(F_{\js}) D_j$ to get
\begin{align}
0 & = W_t - F_{\js} (D_j W) D_i \fe_{\is} - F_{\js}WD_j D_i \fe_{\is} \\
 & = W_t - F_{\js}W_j D_i \fe_{\is} + F_{\js}W \Ric_M(e_j, e_k) \fe_{k^*}- F_{\js}WD_i D_j \fe_{\is} \label{eq:9_35}
\end{align}
We then move around the derivative $D_i$ on the last term of \eqref{eq:9_35} and use \eqref{eq:9_22} to get
\begin{equation}
W_t - W D_i(D_k W \fe_{\ks\is}) + \Lambda W = W_j F_{\js}D_i\fe_{\is}-(\frac{1}{W}+\epsilon)\Ric_M(D \ue, D \ue),
\end{equation}
where $\Lambda = D^2_{ik}\ue D^2_{jl}\ue F_{\ks\js} \fe_{\l^*\is} \geq 0$.

Consider the case that $\Ric_M \geq 0$. Now, taking a cut-off function $\eta \in H^1_0(\Omega)$ and $\eta \geq 0$, we multiply by $W^{-1} \eta$, integrate by parts, use that $|W_j F_{\js}| = (1/W) |\la D\ue, D W\ra|\leq |DW|$, use inequality \eqref{eq:9_1}, and use equation \eqref{eq:9_19} to get
\begin{equation}
\int\limits_\Omega \frac{W_t}{W} \eta + W_k  \fe_{\ks\is} D_i \eta + \Lambda \eta \leq
C\int\limits_\Omega \frac{|\nabla W|}{W}\eta.
\end{equation}

We will have need for a special class of test functions. Fix a point $p_0 \in \Omega$, $t_0 \geq 0$, $\sigma \geq 0$, and $T$ such that $t_0+\sigma < T$. We use $G(\rho,\sigma)$ to denote that class of non-negative test functions $\zeta(x,t)$ defined by:
\begin{enumerate}
\item $\zeta(x,t) = \omega(x)\chi(t)$ where $ \omega \in \Lip(\Omega \cap B(\rho,x_0))$ and $\chi \in \Lip[0,T]. $
\item $\omega \equiv 1$ on $B(\rho/2, x_0)$ and $\suppt \omega \subset \bar B(\rho)$.
\item \begin{enumerate}
	\item if $\sigma > 0$ then $\chi \equiv 1$ on $[\sigma, T]$ and $\chi(0) = 0$.
	\item if $\sigma = 0$ then $\chi \equiv 1$ on $[0, T]$.
	\end{enumerate}
\end{enumerate}
\begin{figure}
\def\svgwidth{5 in}
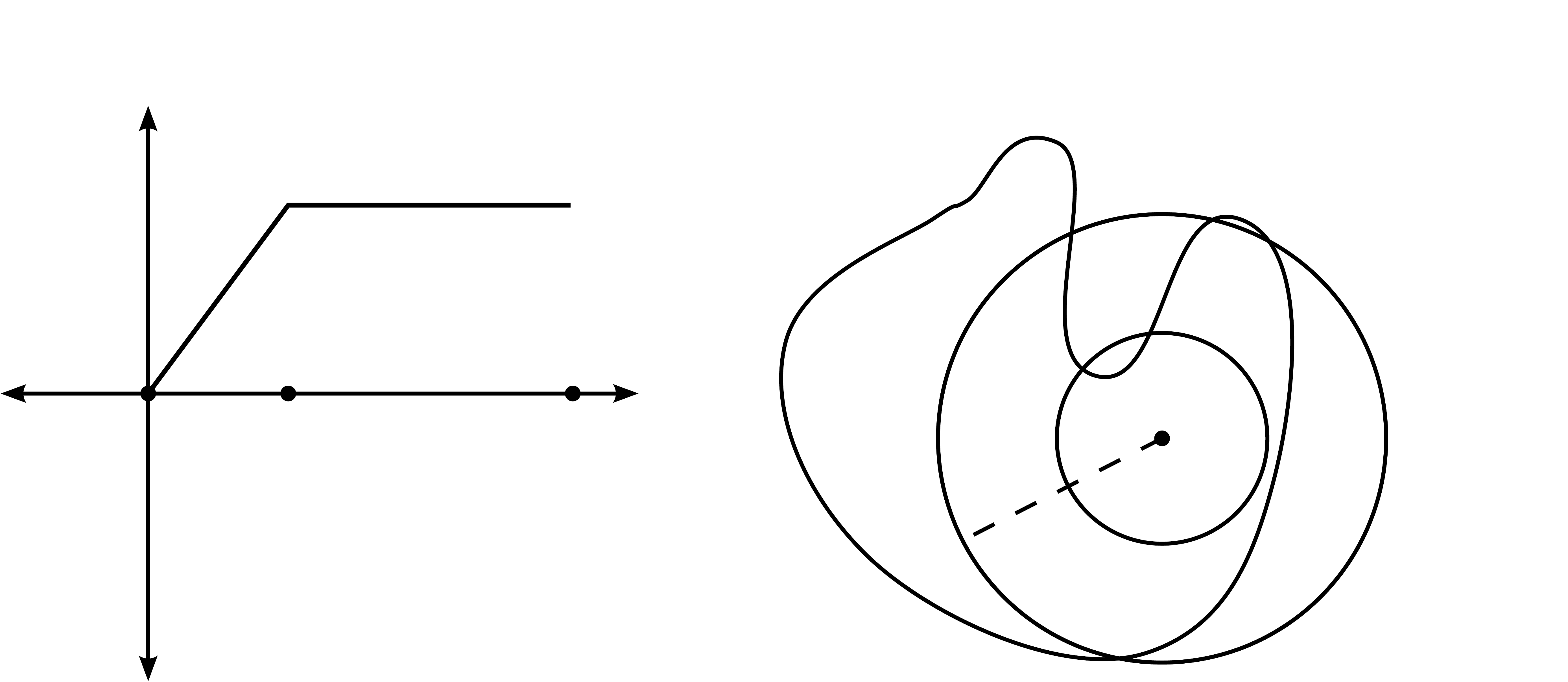
\end{figure}

For a given $\zeta \in G(\rho, \sigma)$ and for a constant $M$ we use the test function $\eta = (W-M)_+ \zeta^2$. Now, we choose another constant $M^*$ such that for $M \geq M^*$, we are guaranteed that $\eta = 0$ on $\Omega \times \{0\}$. So, we let
\begin{equation}
M^*  = \begin{cases}
\sup\limits_{\Omega\times\{0\}} (1+|Du_0|^2)^{1/2} & \sigma = 0 \\
1 & \sigma > 0.
\end{cases}
\end{equation}
Now we use the test function $\eta = (W - M)_+ \zeta^2$ for some $\zeta \in G(\rho, \sigma)$ and $M \geq M^*$. Let $A_k(t) = \{ x \in \Omega\cap B(\rho) | W(x,t) > M \}.$ We get
\begin{align}
\intAMt \frac{W_t (W-M)}{W}\zeta^2 + \fe_{\is\js}W_iW_j \zeta^2 + \Lambda (W-M)\zeta^2 \leq\\
 \intAMt C\frac{|\nabla W|}{W}(W - M) \zeta^2 - \intAMt 2 W_k \fe_{\ks\is} (W-M) \zeta \zeta_i.
\end{align}

Let $\PhiFn(w,k) = \int\limits_k^w \frac{(y-k)_+}{y} \dy$. From $\fe_{\is\js}W_iW_j = |\nabla W|^2 / W + \epsilon |DW|^2$ and $\Lambda \geq 0$, we get
\begin{align}
\intAMt \frac{\partial}{\partial t} \left[\PhiFn(W,t) \zeta^2\right]+ (|\nabla W|^2/W + \epsilon |DW|^2)\zeta^2 \leq \intAMt C\frac{|\nabla W|}{W}(W - M) \zeta^2 \\
 - \intAMt 2 W_k \fe_{\ks\is} (W-M) \zeta \zeta_i - 2\PhiFn(W,t)\zeta\zeta_t.
\end{align}

Now, we estimate $ C |\nabla W| (W-M) \leq (1/4) |\nabla W|^2 + C^2 (W-M)^2,$ and also $ |2W_k \fe_{\ks\is}(W-M)\zeta\zeta_i | \leq (1/4) \fe_{\ks\is}W_k W_i \zeta^2 + 4|D\zeta|^2 (W^{-1} + \epsilon)(W - M)^2$. Therefore,
\begin{align}\label{eq:9_26}
\intAMt \frac{\partial}{\partial t} \left[\PhiFn(W,t) \zeta^2\right]+ \frac{1}{2}(|\nabla W|^2/W + \epsilon |DW|^2)\zeta^2 \leq \intAMt C^2 \frac{(W-M)^2}{W} \zeta^2 \\
+\intAMt 4|D\zeta|^2 (W^{-1} + \epsilon)(W - M)^2+ 2\PhiFn(W,t)\zeta\zeta_t.
\end{align}

For now, let $\epsilon$ be small enough such that $\epsilon^{-1} \geq M^*$. Also, set $M_0 = \epsilon^{-1}$. We have that $(\epsilon W)^{-1} \leq (\epsilon M_0)^{-1} = 1$ on the set $A_M(t)$. Dividing \eqref{eq:9_26} by $\epsilon$, integrating over $[0, t]$, and individually estimating terms on the left hand side, we get
\begin{align}
M_0 \sup\limits_{[0,T]} \int\limits_{A_M(t)} \PhiFn(W, M) \zeta^2 + \frac{1}{2} \int\limits_{0}^T\int\limits_{A_M(t)} |DW|^2 \zeta^2 \dx \dt \\
\leq C^2 \int\limits_0^T\int\limits_{A_M(t)}(W-M)^2 \zeta^2 \dx + \int\limits_0^T\int\limits_{A_M(t)} 8 |D\zeta|^2 (W-M)^2 + 2 M_0\PhiFn(W,M) \zeta \zeta_t.
\end{align}

Therefore, setting $\PsiFn(W,M) = M \PhiFn(W,M)$, we have the following:
\begin{lemma} \label{lemma9_3}
Consider any $\eta \in G(\rho, \sigma)$. Let
\[M^* =
\begin{cases}
 \sup\limits_{\Omega\times\{0\}} (1+|Du_0|^2)^{1/2} & \sigma = 0, \\
1 & \sigma > 0.
\end{cases}
\]
Let $\epsilon_0 = 1/M^*$. For $\epsilon < \epsilon_0$, $M_0(\epsilon) = \epsilon^{-1}$, and $M \in [M_0, 2M_0]$ we have that
\begin{align}\label{eq:9_30}
\sup\limits_{[0,T]} \int\limits_{A_M(t)} \PsiFn(W, M) \zeta^2 + \int\limits_{0}^T\int\limits_{A_M(t)} |DW|^2 \zeta^2 \dx \dt \\
\leq D \int\limits_{0}^T\int\limits_{A_M(t)}(W-M)^2 (\zeta^2 +|D\zeta|^2) \dx + \PsiFn(W,M) |\zeta \zeta_t|.
\end{align}
Here $D = D(u_0, \phi, \Omega)$.
\end{lemma}

Furthermore, by estimating $\sqrt{1+r^2} \leq 1 + (1/2) r^2$, we find that
\begin{equation}
\| (W - M_0)_+ \|_{2, \tilde Q} \leq [C_0 T\epsilon]^{1/2} M_0.
\end{equation}
Therefore, given any $\Theta$, we may find $\epsilon_0 = \epsilon_0(T, \Theta, C_0)$ such that
\begin{equation}\label{eq:9_33}
\| (W - M_0)_+ \|_{2, \tilde Q} \leq \Theta M_0.
\end{equation}

We would like to apply something similar to \cite[Lemma 4.1]{OU1993}, which gives pointwise bounds for functions $W$ that satisfy integration bounds of the form \ref{lemma9_3}.  Oliker-Ural'tseva \cite{OU1993} derive so using an integration iteration estimate and a Sobolev Inequality of the form $\int_\Omega g^2 \dx \leq \beta |\suppt g|^{2/n} \int_\Omega |Dg|^2 \dx.$ Unfortunately, in the setting of $M\times\mathbb R$, for $\Omega' \subset\subset \Omega$, we need to cover $\Omega'$ by geodesic balls $B$ of small enough size such that the in local coordinates $\sigma_{ij} = \delta_{ij} + O(r^2)$ and $|\sigma_{ij,k}| = O(r)$. By choosing the geodesic balls small enough to get uniform estimates of this nature, we may apply the euclidean Sobolev inequality on each geodesic neighborhood. Since $\Omega'$ is compact we can recover the result of Oliker-Ural'tseva\cite{OU1993}, but the dependencies of the constants are a little weaker. Using this, \eqref{eq:9_33}, Lemma \ref{lemma9_3}, and \cite[Lemma 4.1]{OU1993}, we have:

\begin{lemma}
Let $\Ric M \geq 0.$ Let $\Omega' \subset\subset \Omega$ and $T > 0$, $0 \leq \sigma < T$. There exists
\begin{equation}
\epsilon_0 = \begin{cases}
\epsilon_0(T,\sigma, \Omega', u_0, \phi) & \sigma > 0, \\
\epsilon_0(T,\max\limits_\Omega |Du_0|, \Omega', u_0, \phi) & \sigma = 0,
\end{cases}
\end{equation}
such that for $\epsilon \leq \epsilon_0$ and in $\Omega' \times [\sigma, T]$ we have
\begin{equation}
W\epsilon \leq C
\end{equation}
where \begin{equation}
C = \begin{cases}
C(T,\sigma, \Omega', u_0, \phi) & \sigma > 0 \\
C(T,\max\limits_\Omega |Du_0|, \Omega', u_0, \phi) & \sigma = 0.
\end{cases}
\end{equation}
\end{lemma}

Following the approach of Oliker-Ural'tseva\cite{OU1993} and using the estimate $\epsilon |D\ue| \leq C$, it is now possible to obtain a Sobolev inequality of surfaces for functions of compact support on the graph of $\ue$. The facts that $\frac{\ue_t}{W} = D_i \fe_{\is}$, $|\fe_{\is} | \leq 1 + C$, and $|\frac{u_t}{W}| \leq C$ allow us, on small geodesic balls $B$, to apply the Sobolev Inequality of Ladyzhenskaya-Ural'tseva\cite{LU} to get
\begin{equation}
\int\limits_\Se f^2 \, d H_n \leq \beta H_n^{2/n}(\suppt f \cap \Se) \int\limits_\Se |\nabla f|^2 \, d H_n
\end{equation}
where $f \in C^1_0(B)$ and $\beta = \beta(C)$. The proof of Theorem 2.5 in Oliker-Ural'tseva\cite{OU1993} may now be carried out on small geodesic balls to obtain

\begin{theorem}
Let $\Omega' \subset\subset \Omega$ and $T> 0$. There exists a constant $C = C(\Omega', T, u_0, \phi)$ such that
\begin{equation}
\sup\limits_{\Omega'\times [0,T]} |D\ue| \leq C
\end{equation}
for $\epsilon < \epsilon_0 = \epsilon(\Omega', T, u_0, \phi)$.
\end{theorem}

From here, by standard parabolic theory, we get estimates on the other spatial derivatives of $\ue$.

\begin{corollary}
Let $\Omega' \subset\subset \Omega$ and $T > 0$. There exists $\epsilon_0 = \epsilon_0(\Omega', T, u_0, \phi)$ such that for every $\epsilon < \epsilon_0$ we have
\begin{equation}
\sup\limits_{\Omega'\times[0,T]} |D^\alpha \ue | \leq C
\end{equation}
where $C = C(\Omega', T, u_0, \phi, \alpha)$.
\end{corollary}

From these estimates we see that there is a sequence $\epsilon_i \to 0$ such that on compact sets $\Omega' \subset\subset\Omega$ we have that $\ue$ and its derivatives converge uniformly to a function $u$ and its derivatives.

From the estimates, we have that there exists a sequence $t_i \to \infty$ such $u(\cdot, t_i)$ and its derivatives converge uniformly on compact subset $\Omega' \subset\subset \Omega$ to a function $\bar u_1$. Consider another such $s_i$ with limit $\bar u_2$. We show that $\bar u_1= \bar u_2$. Note, Oliker-Ural'tseva\cite{OU1997} show this for $\Omega \subset \mathbb R^n$ by working directly with the function $u$, using that minimizers of the functional $J$ differ by a constant, and the part of the boundary that is mean convex still has suitable barriers so we do get uniqueness. For the case of general $M$, we may have that the boundary has no mean convex part. For example, consider a large enough ball in a sphere $S^n$.

We show that there is uniqueness depending on the choice of the sequence $\epsilon_i$, but not depending on the choice of sequence in time $t_i$. To do this, we use the uniqueness of the limit for $u^\epsilon$ that comes from the uniqueness of minimizers for a perturbed functional. First, a discussion of this uniqueness for $\ue$.

Similar to Oliker-Ural'tseva\cite{OU1997}, we have that $\ue$ satisfies an estimate uniform in $\epsilon$:
\begin{equation}\label{eq:9_41}
\int\limits_0^\infty\int\limits_\Omega \frac{|\ue_t|^2}{\sqrt{1+|D\ue|^2}} \dx \dt \leq C,
\end{equation}
where $C = C(u_0, \Omega, \phi)$.

Consider the functional
\begin{equation}
E^{\epsilon,t}(w) = \int\limits_\Omega \fe(x,Dw) + f^{\epsilon,t} w \dx
\end{equation}
where $ f^{\epsilon,t} = \frac{\ue_t(t)}{\sqrt{1+|D\ue(t)|^2}}$. Note, by the convexity of $E^{\epsilon, t}$, we have that $\ue(\cdot, t)$ is a minimizer of $E^{\epsilon,t}$ for functions $w \in W^{1,2}(\Omega)$ with $w = \phi$ on $\partial \Omega$. Note, also that from \eqref{eq:9_41}, we have for some sequence $t_k$ that $\int\limits_\Omega |f^{\epsilon,t_k}|^2 \dx \to 0$.

Since the function $\ue$ satisfies a uniformly parabolic equation, we know that its derivatives are bounded uniformly on $\Omega \times (0,\infty)$ (but of course the bound may depend on $\epsilon$). Therefore, for any sequence $t_i$, we may pass to a subsequence such that $\ue(\cdot, t_i) \to \bar u^\epsilon (\cdot, t_i)$ uniformly on $\Omega$. Note, also that from \eqref{eq:9_41}, we have for some sequence $t_k$ that $\int\limits_\Omega |f^{\epsilon,t_k}|^2 \dx \to 0$.

Now consider the convex functional
\begin{equation}
E^\epsilon (w) = \int\limits_\Omega \fe(x, Dw) \dx.
\end{equation}
First, we again pass to a subsequence such that $D\ue(\cdot, t_k) \to D\bar u^\epsilon$ uniformly. Since we have uniform convergence, for any $w \in W^{1,2}(\Omega)\cap L^\infty(\Omega)$ with $w = \phi$ on $\partial\Omega$ we have that
\begin{equation} \label{eq:9_64}
E^\epsilon(\bar u^\epsilon) = \lim\limits_{k\to\infty} E^\epsilon(\ue(\cdot, t_k)) = \lim\limits_{k\to\infty} E^{\epsilon, t_k}(\ue(\cdot, t_k)) - \int\limits_\Omega f^{\epsilon,t_k}\ue
\end{equation}
\begin{equation}
 = \lim\limits_{k\to\infty} E^{\epsilon, t_k}(\ue(\cdot, t_k)) \leq \lim\limits_{k\to\infty} E^{\epsilon,t_k}(w) \to E^\epsilon(w).
\end{equation}

Therefore, the limit $\bar u^\epsilon$ is a minimizer of $E^\epsilon$ which implies it is unique. Now, we may apply this to discuss the uniqueness of $\bar u$.

\begin{theorem}
Let $\Ric_M \geq 0$ and let $\Omega \subset M$ be a bounded domain. There exists a sequence $\epsilon_i \to 0$ such that the solutions $u^{\epsilon_i}$ of \eqref{eq:emcf} converges uniformly to a function $u$ in $C^\infty(K)$ for compact $K \subset\subset \Omega\times (0,\infty)$. We have that $u\in C^\infty(\Omega\times (0,\infty)) \cup L^\infty([0,\infty); W^{1,1}(\Omega)).$

Consider any sequence $t_i$ such that $u(\cdot, t_i)$ converges uniformly to a function $\bar u$ in $C^\infty(K)$ for compact $K \subset\subset \Omega$. We have that $\bar u$ is a generalized solution to the Dirichlet problem with boundary data $\phi(x)$. Furthermore, $\bar u$ depends only on the choice of sequence $\epsilon_i$; $\bar u$ is independent of the choice of $t_i$.
\end{theorem}

\begin{proof}
The only thing left to discuss is uniqueness. Consider another sequence $s_i$ such that $\ue(\cdot,s_i) \to \bar v^\epsilon$ and their spatial derivatives converge uniformly on compact subsets $\Omega' \subset\subset\Omega$. For any $x \in \Omega$ and $\delta > 0$, we have for some $\epsilon_j$ that

\begin{equation}
|\bar u(x) - \bar v(x)| \leq 2\delta + |u(x,t_i) - u(x,s_i)| \leq 4\delta + |u^{\epsilon_j}(x,t_i) - u^{\epsilon_j}(x,s_i)|.
\end{equation}
We may pass to a subsequence such that $u^{\epsilon_j}(\cdot, t_i) \to \bar u^{\epsilon_j}$ and $u^{\epsilon_j}(\cdot, s_i) \to \bar v^{\epsilon_j}$ uniformly on compact subsets of $\Omega' \subset\subset \Omega$. From our discussion above, we know that $\bar u^{\epsilon_j} = \bar v^{\epsilon_j}$, therefore we get $|\bar u(x) - \bar v(x)| \leq 4\delta$. Hence, $\bar u = \bar v$.

\end{proof}


\bibliographystyle{plain}
\bibliography{MR_Ref}

\begin{thebibliography}{10}

\bibitem{Car}
Manfredo~Perdig{\~a}o do~Carmo.
\newblock {\em Riemannian geometry}.
\newblock Mathematics: Theory \& Applications. Birkh\"auser Boston Inc.,
  Boston, MA, 1992.
\newblock Translated from the second Portuguese edition by Francis Flaherty.

\bibitem{GT}
David Gilbarg and Neil~S. Trudinger.
\newblock {\em Elliptic partial differential equations of second order}, volume
  224 of {\em Grundlehren der Mathematischen Wissenschaften [Fundamental
  Principles of Mathematical Sciences]}.
\newblock Springer-Verlag, Berlin, second edition, 1983.

\bibitem{Giu}
Enrico Giusti.
\newblock {\em Minimal surfaces and functions of bounded variation}, volume~80
  of {\em Monographs in Mathematics}.
\newblock Birkh\"auser Verlag, Basel, 1984.

\bibitem{GK}
Sigmundur Gudmundsson and Elias Kappos.
\newblock On the geometry of tangent bundles.
\newblock {\em Expo. Math.}, 20(1):1--41, 2002.

\bibitem{HS}
David Hoffman and Joel Spruck.
\newblock Sobolev and isoperimetric inequalities for {R}iemannian submanifolds.
\newblock {\em Comm. Pure Appl. Math.}, 27:715--727, 1974.

\bibitem{Kow}
Old{\v{r}}ich Kowalski.
\newblock Curvature of the induced {R}iemannian metric on the tangent bundle of
  a {R}iemannian manifold.
\newblock {\em J. Reine Angew. Math.}, 250:124--129, 1971.

\bibitem{LU}
O.~A. Ladyzhenskaya and N.~N. Ural'tseva.
\newblock Local estimates for gradients of solutions of non-uniformly elliptic
  and parabolic equations.
\newblock {\em Comm. Pure Appl. Math.}, 23:677--703, 1970.

\bibitem{Lie}
Gary~M. Lieberman.
\newblock {\em Second order parabolic differential equations}.
\newblock World Scientific Publishing Co. Inc., River Edge, NJ, 1996.

\bibitem{Mir}
Michele Miranda, Jr.
\newblock Functions of bounded variation on ``good'' metric spaces.
\newblock {\em J. Math. Pures Appl. (9)}, 82(8):975--1004, 2003.

\bibitem{OU1993}
Vladimir~I. Oliker and Nina~N. Uraltseva.
\newblock Evolution of nonparametric surfaces with speed depending on
  curvature. {II}. {T}he mean curvature case.
\newblock {\em Comm. Pure Appl. Math.}, 46(1):97--135, 1993.

\bibitem{OU1997}
Vladimir~I. Oliker and Nina~N. Ural'tseva.
\newblock Long time behavior of flows moving by mean curvature. {II}.
\newblock {\em Topol. Methods Nonlinear Anal.}, 9(1):17--28, 1997.

\bibitem{SW}
Friedmar Schulz and Graham Williams.
\newblock Barriers and existence results for a class of equations of mean
  curvature type.
\newblock {\em Analysis}, 7(3-4):359--374, 1987.

\bibitem{Sim}
Leon Simon.
\newblock {\em Lectures on geometric measure theory}, volume~3 of {\em
  Proceedings of the Centre for Mathematical Analysis, Australian National
  University}.
\newblock Australian National University Centre for Mathematical Analysis,
  Canberra, 1983.

\bibitem{Spr}
Joel Spruck.
\newblock Interior gradient estimates and existence theorems for constant mean
  curvature graphs in {$M^n\times\bold R$}.
\newblock {\em Pure Appl. Math. Q.}, 3(3, Special Issue: In honor of Leon
  Simon. Part 2):785--800, 2007.

\end{thebibliography}

\end{document}